\documentclass[a4paper,12pt]{article}
\usepackage[utf8]{inputenc}
\usepackage[cm]{aeguill}
\usepackage{amsfonts}
\usepackage{latexsym}
\usepackage[dvips]{graphicx}

\usepackage{amsmath}
\usepackage{amssymb}
\usepackage{amsthm}
\usepackage{bbm}
\usepackage{nicefrac}
\usepackage[T1]{fontenc}
 
\newtheorem{df}{Definition}[section]
\newtheorem{lm}[df]{Lemma}
\newtheorem{pr}[df]{Proposition}
\newtheorem{Th}[df]{Theorem}
\newtheorem{co}[df]{Corollary}
\newtheorem{rem}[df]{Remark}
\newtheorem{ex}[df]{Example}

\newcommand{\e}{\varepsilon}

\newcommand{\R}[1]{\mathbb{R}^{#1}}
\newcommand{\T}[1]{(T_{c}^{-})^{#1}}
\newcommand{\ttt}{T_c^-}
\newcommand{\tttt}{\widehat{T}_c^-}
\newcommand{\TT}[1]{(T_{c}^+)^{#1}}
\newcommand{\TTT}{T_c^+}

\newcommand{\<}{\prec}
\newcommand{\s}{\mathcal{S}}

\newcommand{\A}{\mathcal{A}}
\renewcommand{\AA}{\widetilde{\mathcal{A}}}
\newcommand{\HH}{\mathcal{H}}
\newcommand{\HHH}{\widehat{\mathcal{H}}}
\newcommand{\hh}[1]{\mathcal{H}(#1)\cap C^0(X,\R{})}
\newcommand{\om}{\omega}
\newcommand{\XZ}{X^\mathbb{Z}}
\newcommand{\fe}{\varphi_{\varepsilon}}

\renewcommand{\d}{\operatorname{d}}

\newcommand{\diam}{\operatorname{diam}}

\author{M. ZAVIDOVIQUE \\ \small \texttt {UMPA, ENS Lyon, 46 all\' ee d'Italie, 69007, Lyon, France} \\ \small \texttt{e-mail: maxime.zavidovique@umpa.ens-lyon.fr}}
\title{Strict sub-solutions and Ma\~ne potential in discrete  weak KAM theory}

\begin{document}

\maketitle

\begin{abstract}
In this paper, we explain some facts on the discrete case of weak KAM theory. In that setting, the Lagrangian is replaced by a cost $c:X\times X \rightarrow \mathbb{R}$, on a ``reasonable'' space $X$. This covers for example the case of periodic time-dependent  Lagrangians. As is well known, it is possible in that case to adapt most of weak KAM theory. A major difference is that critical sub-solutions are not necessarily continuous. We will show how to define a Ma\~ne potential. In contrast to the Lagrangian case, this potential is not continuous. We will recover the Aubry set from the set of continuity points of the Ma\~ne potential, and also from critical sub-solutions.
\end{abstract}
\newpage
\section*{Introduction}
In the past twenty years, new techniques have been developed in order to study time-periodic or autonomous Lagrangian dynamical systems. Among them, Aubry-Mather theory (for an introduction see \cite{Ba} for the annulus case and \cite{Mat}, \cite{Ma} for the compact, time periodic case) and Albert Fathi's weak KAM theory (see \cite{Fa} for the compact case and \cite{FaMa} for the non-compact case) have appeared to be very fruitful. More recently, a discretization of weak KAM theory applied to optimal transportation has allowed to obtain deep results of existence of optimal transport maps (see for example \cite{Be},\cite{fatfig07}). A quite similar formalism was also used in the study of time periodic Lagrangians, for example in (\cite{cis} or \cite{mas}). In this paper, we give analogue results in this discrete setting of those already obtained in the continuous one. In particular, our phase space $X$ will be required to have very little regularity (for example a length space with compact closed balls will do) and no global compactness assumption.
\\
\\
In a first part we introduce the Lax-Oleinik semi-groups $\ttt$ and $\TTT$ and study its sub-solutions. We start with a cost $c:X^2\rightarrow \R{}$ continuous which verifies: 
\begin{enumerate}
 \item\label{unif} \textbf{Uniform super-linearity}: for every $k\geqslant 0$,
   there exists $C(k)\in \R{}$ such that 
   $$\forall (x,y)\in X^2,
   c(x,y)\geqslant k\d(x,y)-C(k);$$
\item \label{unifb} \textbf{Uniform boundedness}: for every $R\in \R{}$, there
  exists $A(R)\in \R{}$ such that $\d(x,y)\leqslant R \Rightarrow
  c(x,y)\leqslant A(R)$.
\end{enumerate}
A function $u$ is an $\alpha$-sub-solution for $c$ if
\begin{equation}\label{sub}
\forall (x,y)\in X^2,u(y)-u(x)\leqslant c(x,y)+\alpha.
\end{equation}
The critical constant $\alpha[0]$ is the smallest constant $\alpha$ such that there are $\alpha$-sub-solutions. In the first part we prove, as in \cite{FaSi}, the existence of critical sub-solutions which are strict on a maximal set:
\begin{Th}\label{Strict}
 There is a continuous function $u_1:X\rightarrow \R{}$ which is an $\alpha[0]$-sub-solution such that for every $(x,y)\in X^2$, if there exists an $\alpha[0]$-sub-solution, $u$ such that
$$u(y)-u(x)<c(x,y)+\alpha[0],$$
then we also  have
$$u_1(y)-u_1(x)<c(x,y)+\alpha[0].$$
\end{Th}
The proof is done using the Lax-Oleinik semi-groups $\ttt$ and $\TTT$ and the notion of Aubry set as introduced in \cite{Be}.
\\ \\
The second part is devoted to the study of the continuity of sub-solutions and of an analogue of Ma\~ne's potential. Those two problems are closely related. As a matter of fact, in the Lagrangian continuous case, all critical sub-solutions are equi-Lipschitz maps and the Aubry set may be defined as the set of points $x\in X$ such that any sub-solution is differentiable at $x$. Moreover, this information is encrypted in the Ma\~ne potential $\phi :X^2\rightarrow \R{}$. more precisely, Fathi and Siconolfi (\cite{FaSi}) proved that a point $x$ is in the projected Aubry set if and only if the function $\phi_x:y\mapsto \phi(x,y)$ is differentiable at $x$. In the discrete case, we will see that sub-solutions are not necessarily continuous. However, analogously to the continuous case, the projected Aubry set is the set of points where all sub-solutions are continuous. Moreover, our Ma\~ne potential will verify the following: 
\begin{Th}\label{Mane}
 There in a function $\varphi :X^2\rightarrow \R{}$ which satisfies the following:
\begin{enumerate}
\item [(1)] for any $x\in X$, $\varphi(x,x)=0$;
\item[(2)] a function $u$ is a critical sub-solution if and only if
$$\forall (x,y)\in X^2, u(y)-u(x)\leqslant \varphi(x,y);$$
\item[(3)] for any $x\in X$, the function $\varphi_x:y\mapsto \varphi(x,y)$ is a critical sub-solution;
\item[(4)] a non isolated point $x\in X$ is in the Aubry set if and only if the function $\varphi_x:y\mapsto \varphi(x,y)$ is continuous at $x$;
\item[(5)] if $x\in X$ is non isolated, the function $\varphi_x$ is continuous at $x$ if and only if it is a  negative weak KAM solution, that is a fixed point of $\ttt +\alpha[0]$.
\end{enumerate}
\end{Th}
For the definition of the semi-group $\ttt$ see section 1.

\section*{Acknowledgment}
I would like to thank Albert Fathi without whom this article would never have been written. His remarks and comments were of  invaluable help. This paper was partially elaborated during a stay at the Sapienza University in Rome. I wish to thank Antonio Siconolfi, Andrea Davini and the Dipartimento di Matematica "Guido Castelnuovo" for their hospitality while I was there. I also would like to thank Explora'doc which partially supported me during this stay. Finally, I would like to thank the ANR KAM faible (Project BLANC07-3\_187245, Hamilton-Jacobi and Weak KAM Theory) for its support during my research.

\section{On critical sub-solutions}
In this section we will fix a metric space $X$ which is a $B$-length space at scale $K$ for some constants $B$ and $K$ (see \ref{scale} for the exact definition) with compact closed balls and let $c:X\times
X\rightarrow \R{}$ be a continuous function which is uniformly
super-linear and uniformly bounded that is which verifies condition 1 and 2 of the introduction.
 
\begin{df}\rm
If $\alpha\in \R{}$ and $u:X\rightarrow \R{}$ is a (not necessarily continuous) function, we will say that $u$ is
$\alpha$-dominated (in short $u\<c+\alpha$) if
$$\forall (x,y)\in X^2,u(y)-u(x)\leqslant c(x,y)+\alpha.$$
We will denote by $\HH(\alpha)$ the set of $\alpha$-dominated functions. 

Following Albert Fathi's
weak KAM theory we introduce the Lax-Oleinik semi-groups:
$$\ttt u(x)=\inf_{y\in X}u(y)+c(y,x);$$
$$\TTT u(x)=\sup_{y\in X}u(y)-c(x,y).$$
\end{df}

\begin{Th}[weak KAM]\label{kam}
 There is a constant $\alpha [0]$ such that the equation $u=\ttt u+\alpha [0]$ (resp. $u=\TTT u-\alpha [0]$) admits a continuous solution and such that $\HH(\alpha)$ is empty for $\alpha<\alpha [0]$.
\end{Th}

\begin{proof}
 see the end of the appendix (\ref{proof kam}).
\end{proof}

We say
that a function $u$ is critically dominated or that it is a critical
sub-solution if it is $\alpha[0]$-dominated.  
Finally, we call negative
(resp. positive) weak KAM solution a fixed point of the operator $\ttt
+\alpha[0]$ (resp. $\TTT -\alpha[0]$). Let us state that weak KAM solutions exist by (\ref{kam}).  The following proposition is a
direct consequence of the definitions:
\begin{pr}
A function $u$ is a critical sub-solution if and only if it verifies one of the following properties:
\begin{enumerate}
 \item[(i)]$\forall (x,y)\in X^2,u(x)-u(y)\leqslant c(y,x)+\alpha [0]$ (or $u\<c+\alpha [0]$);
\item[(ii)] $u\leqslant \ttt u+\alpha [0]$;
\item[(iii)] $u\geqslant \TTT u-\alpha [0]$.
\end{enumerate}

\end{pr}
 The more analytical denomination of sub-solution is useful because it allows to introduce the notion of being strict at some point:
\begin{df}
 \rm Consider $x_0\in X$ and $u\<c+\alpha [0]$ a critical sub-solution. 
We will say that $u$ is strict at $(x,y)\in X^2$ if and only if 
$$u(x)-u(y)< c(y,x)+\alpha [0].$$
We will say that $u$ is strict at $x\in X$ if
\begin{equation*}
\forall y\in X,  u(y)-u(x)<c(x,y)+\alpha [0] \; 
\mathrm{ and }\;
u(x)-u(y)<c(y,x)+\alpha [0]  .
\end{equation*}

\end{df}

We first give a characterization of continuous strict sub-solutions.
 \begin{pr}\label{strict}
A \textbf{continuous} sub-solution $u$ is strict at $x$ if and only if
$u(x)<\ttt u(x)+\alpha [0]$ and $u(x)>\TTT u(x)-\alpha [0]$.
 \end{pr}
 \begin{proof}

By definition, if $u$ is strict at $x$ then 
$$\forall y\in X,
u(x)-u(y)< c(y,x)+\alpha [0].$$
 In the appendix (\ref{HH} and \ref{apriori}), it is shown that the function
$y\mapsto c(y,x)+\alpha [0]-u(y)+u(x)$ tends to $+\infty$ when $\d(x,y)$ tends to 
$+\infty$. Since closed balls are compact, by continuity of $u$, the infimum in the definition of $\ttt$ is achieved. Therefore we must have 
$$u(x)<\ttt
u(x)+\alpha [0].$$
  Similarly, if for every $y\in X$,
$u(y)-u(x)<c(x,y)+\alpha [0]$ then
$$u(x)>\sup_{y\in X} u(y)-c(x,y)-\alpha [0]=\TTT u(x)-\alpha [0].$$
 The converse is clear.
\end{proof}
Before going any further, let us give some definitions:
\begin{df}\label{d:aub2}\rm
 Let $u:X\rightarrow \R{}$ verify $u\<c+\alpha [0]$. We will say that
 a chain $(x_i)_{0\leqslant i\leqslant n}$ of points in $X$ is $(u,c,\alpha
 [0])$-calibrated if 
$$u(x_n)=
 u(x_0)+c(x_0,x_1)+\cdots+c(x_{n-1},x_n)+n\alpha [0].$$
 Notice that a
 sub-chain formed by consecutive elements of a calibrated chain is again calibrated since $u\<c+\alpha[0]$.  

Following Bernard and Buffoni \cite{Be} we will call
 Aubry set of $u$, the subset $\AA_u$ of $\XZ$ consisting of the
 sequences whose finite sub-chains are $(u,c,\alpha
 [0])$-calibrated.  The projected Aubry set of $u$ is
$$\A _u=\{x\in X,\exists (x_n)_{n\in \mathbb{Z}} ,(u,c,\alpha
[0])\textrm{-calibrated with }x_0=x\}. $$
 The Aubry set is
$$\AA=\bigcap_{u\<c+\alpha [0]}\AA _{u}.$$
 The projected Aubry
set is
 $$\A=\bigcap_{u\<c+\alpha [0]}\A _{u},$$
  where in both cases, the
intersection is taken over all critically dominated functions.
\end{df}

We begin by a very simple lemma that will be of great use:

\begin{lm}\label{trivial}
 Let $u\< c+\alpha[0]$ be a critically dominated function and $(x,y)\in X^2$. If the following identity is verified: 
$$u(x)-u(y)=c(y,x)+\alpha[0],$$
then $u(x)=\ttt u(x)+\alpha[0]$.
If the following identity is verified
$$\ttt u(x)-\ttt u(y)=c(y,x)+\alpha[0],$$ 
then $u(y)=\ttt u(y)+\alpha[0]$ and $\ttt u(x)=u(y)+c(y,x)$.
\end{lm}

\begin{proof}
 The first part is straightforward from the definitions.
For the second point write
$$\ttt u(x)=\ttt u(y)+c(y,x)+\alpha [0]\geqslant
 u(y)+c(y,x)\geqslant \ttt u(x)$$
therefore, all inequalities must be equalities which proves the lemma.
\end{proof}

The following lemma, along with the fact that the image by the Lax-Oleinik semi-group of a dominated
function is continuous (cf. \ref{HH}), show that all
the intersections in the definitions of the Aubry sets and projected Aubry sets may be taken on continuous functions.
\begin{pr}\label{egalite aubry}
 Let $u\<c+\alpha [0]$ be a dominated function, then $\AA_u=\AA_{\ttt u}$.
In particular, we also have
$\A_u=\A_{\ttt u}.$
\end{pr}
\begin{proof}
 First we prove the inclusion $\AA_u\subset \AA_{\ttt u}$. Let us consider the sequence  $(x_n)_{n\in
   \mathbb{Z}}\in \AA_u$. Since $u$ is dominated and the sequence
 $(x_n)_{n\in \mathbb{Z}}\in \AA_u$ is ($u,c,\alpha [0]$)-calibrated we have for all $k\in \mathbb{Z}$
$$u(x_{k+1})= u(x_{k})+c(x_{k},x_{k+1})+\alpha [0],$$
 therefore lemma \ref{trivial} yields
$$\forall k\in \mathbb{Z},\ttt u(x_{k+1})+\alpha [0]=u(x_{k+1}).$$
 Therefore, the
 sequence $(x_n)_{n\in \mathbb{Z}}$ is ($\ttt u,c,\alpha
 [0]$)-calibrated and belongs to  $\AA_{\ttt u}$.

 We now prove the reverse
 inclusion $\AA_{\ttt u}\subset \AA_u$. Let $(x_n)_{n\in \mathbb{Z}}\in
 \AA_{\ttt u}$. We have that for any $k\in \mathbb{Z}$
$$\ttt u(x_{k+1})=\ttt u(x_{k})+c(x_{k},x_{k+1})+\alpha [0],$$
therefore using the second part of \ref{trivial}
 $$\forall k\in \mathbb{Z}, u(x_{k})=\ttt u(x_{k})+\alpha [0],$$ 
 and the sequence $(x_n)_{n\in \mathbb{Z}}$ is
 ($u,c,\alpha [0]$)-calibrated.
\end{proof}
Here is a lemma that will be useful in the sequel:
\begin{lm}\label{I}
 There is a continuous function $u\<c+\alpha [0]$ such that
 $\AA_u=\AA$.
\end{lm}
\begin{proof}
Let us consider the set $\s=\{u\in
C^0(X,\R{}),u\<c+\alpha [0]\}$ of continuous dominated functions . This set is separable for the compact
open topology so let $(u_n)_{n\in \mathbb{N}^*}$ be a sequence dense
in $\s$. Consider now $(a_n)_{n\in \mathbb{N}^*}$ a sequence of
positive real numbers such that $\sum a_n=1$ and $u=\sum a_nu_n$
converges uniformly on each compact subset of $X$. To construct such a
sequence, one can for example fix an $x_0\in X$ and for any $n>1$, take
$a_n=\min\{2^{-n},1/(2^n\|u_n\|_{\infty,\overline{B(x_0,n)}})\}$
then take $a_1=1-\sum_{n>1}a_n>0$. The function $u$ is clearly
continuous and since $u$ is a convex sum of elements of $\s$, one can
easily verify that $u\in\s$. Moreover, since each $u_n$ is dominated,
if a chain is $(u,c,\alpha [0])$-calibrated then it is $(u_n,c,\alpha
[0])$-calibrated for every $n\in \mathbb{N}^*$.  As a matter of fact,
if 
$$u(x_{n'})-u(x_n)=\sum_{k\in \mathbb{N}^*}
a_k(u_k(x_{n'})-u_k(x_n))=(n'-n)\alpha
[0]+\sum_{i=n}^{n'-1}c(x_i,x_{i+1}),$$
 since for each $k$ the following
inequality holds
$$\forall k\in
\mathbb{N}^*,u_k(x_{n'})-u_k(x_n)\leqslant (n'-n)\alpha
       [0]+\sum_{i=n}^{n'-1}c(x_i,x_{i+1}),$$
 and considering that
       $\sum a_n =1$ and $a_n>0$, the inequalities above must be equalities
$$\forall k\in \mathbb{N}^*,u_k(x_{n'})-u_k(x_n)= (n'-n)\alpha
[0]+\sum_{i=n}^{n'-1}c(x_i,x_{i+1}).$$
 Finally, since the $u_k$ are
dense in $\s$ we obtain
$$\forall u'\in \s,u'(x_{n'})-u'(x_n)= (n'-n)\alpha
[0]+\sum_{i=n}^{n'-1}c(x_i,x_{i+1}).$$ 
Hence such a calibrated chain is calibrated by every element of $\s$. In particular, for every $ u'
\in \s$, we have $\AA_{u}\subset \AA_{u'}$ therefore $\AA_u\subset \AA$. The reverse
inclusion follows from the definition of $\AA$. Similarly, projecting
on $X$, we get that $\A_u=\A$.
\end{proof}
As an immediate consequence we get the following:
\begin{co}
 The following equality holds:
$$\A=p(\AA),$$
 where $p$ denotes the canonical projection from $\XZ$ to $X$.
\end{co}
The following lemma is useful:
\begin{lm}\label{IIII}
 If $u\<c+\alpha [0]$ and $x\in X$ then 
$x$ is in $\A_u$ implies 
$$\forall p\in
 \mathbb{N}, \T{p}u(x)+p\alpha[0]=u(x)=\TT{p}u(x)-p\alpha[0].$$
Moreover, if $u$ is continuous then the converse is true, that is if
$$\forall p\in
 \mathbb{N}, \T{p}u(x)+p\alpha[0]=u(x)=\TT{p}u(x)-p\alpha[0],$$
then $x\in \A_u$.
\end{lm}

\begin{proof}
 If $(x_n)_{n\in\mathbb{Z}}\in \XZ$
 is calibrating for $u$ then for every integer $p$
$$u(x_p)-u(x_0)=p\alpha [0]+\sum_{i=0}^{p-1}c(x_i,x_{i+1}).$$
Therefore, the domination hypothesis gives us that
$$\forall p\in \mathbb{N}^*,\T{p}u(x_0)+p\alpha[0]=u(x_0),$$
and
$$\forall p\in \mathbb{N}^*,\TT{p}u(x_0)-p\alpha[0]=u(x_0).$$

Conversely, let us assume that for every $p\in \mathbb{N}$, 
$$\T{p}u(x)+p\alpha[0]=u(x)=\TT{p}u(x)-p\alpha[0].$$
 then by successive applications of point (iv) of proposition \ref{atteint} we
 can find  chains $(x_{-p}^p,\ldots
 x_{-1}^p,x_0^p=x,x_1^p,\ldots,x_p^p)$ such that
 $$\forall p \in \mathbb{N},\T{p}u(x)= u(x_{-p}^p)+\sum_{i=-p}^{-1}c(x_i^p,x_{i+1}^p),$$
 and
 $$\forall p \in \mathbb{N},\TT{p}u(x)= u(x_{p}^p)-\sum_{i=0}^{p-1}c(x^p_{i},x^p_{i+1}).$$
 Using the assumption we made, we obtain that 
 $$\forall p \in \mathbb{N},u(x)-u(x_{-p}^p)= \sum_{i=-p}^{-1}c(x^p_i,x^p_{i+1})+p\alpha[0],$$
 and
 $$\forall p \in \mathbb{N},u(x_{p}^p)-u(x)= \sum_{i=0}^{p-1}c(x^p_{i},x^p_{i+1})+p\alpha[0].$$
 Summing these two last equalities we get
 $$\forall p \in \mathbb{N},u(x_{p}^p)-u(x_{-p}^p)= \sum_{i=-p}^{p-1}c(x^p_{i},x^p_{i+1})+2p\alpha[0],$$
which proves that the chains $(x_{-p}^p,\ldots
 x_{-1}^p,x_0^p=x,x_1^p,\ldots,x_p^p)$ are calibrating for $u$.
 
  By \ref{apriori}, for every integer $n\in \mathbb{Z}$, the sequence $(x_n^p),p\geqslant |n|$ is bounded hence, by a diagonal extraction ($p_l\rightarrow +\infty$ as $l\rightarrow +\infty$) we can assume each $(x_n^{p_l}),p_l\geqslant |n|$ converges to a $x_n\in X$.
  Let us now fix two integers $m$ and $n$ such that $m\leqslant n$. If $p_l\geqslant |m|,|n|$ we have
  $$u(x_{n}^{p_l})-u(x_{m}^{p_l})= \sum_{i=m}^{n-1}c(x^{p_l}_{i},x^{p_l}_{i+1})+(n-m)\alpha[0],$$
  letting $p_l$ go to $+\infty$, using the continuity of $u$, the following holds
    $$u(x_{n})-u(x_{m})= \sum_{i=m}^{n-1}c(x_{i},x_{i+1})+(n-m)\alpha[0].$$
    Since $m$ and $n$ were taken arbitrarily, this proves that
   the sequence $(x_k)_{k\in \mathbb{Z}}$ is calibrating for $u$ and therefore is the bi-infinite chain that we are looking for.
\end{proof}
Let us define yet another Aubry set :
\begin{df}
Let $S$ from $\XZ$ to $\XZ$ be the shift operator.  We define
$$\widehat{\A}_u =\{(x,y)\in X^2 ,\exists z\in \AA_u ,x=p(z)\; \mathrm{
  and }\; y=p\circ S(z)\},$$ 
and 
$$\widehat{\A} =\{(x,y)\in X^2 ,\exists z\in \AA ,x=p(z)\; \mathrm{
  and }\; y=p\circ S(z)\}.$$
 \end{df}

We are now ready to prove the following theorem, which in particular is stronger than theorem \ref{Strict}. The proof is inspired from the unpublished manuscript \cite{fasi3}.
\begin{Th}\label{SStrict}
 For every sub-solution $u$ there is a continuous sub-solution $u'$ which is strict at every $(x,y)\in X^2-\widehat{\A}_u$ and such that $u=u'$ on $\A_u$.
 There is a continuous sub-solution which is strict at every $(x,y)\in X^2-\widehat{\A}$.
\end{Th}
\begin{proof}
  Replacing $u$ by $\ttt u$ (which does not change the Aubry set by  \ref{egalite aubry}) we can assume that $u$ is continuous. Consider the function 
$$u'=\sum_{n\in
   \mathbb{N}}a_n\T{n}u+\sum_{n\in \mathbb{N}^*}b_n\TT{n}u,$$ 
where the $a_n$ and the $b_n$ are chosen as in the proof of lemma \ref{I}, positive, 
 such that the sums above are convergent for the compact open topology and
 $\sum a_n+\sum b_n =1$.  For the same reasons as in the
 proof of \ref{I}, $u'$ is a continuous and critically dominated
 function.  
Let $(x,y)\in X^2$ verify
 $u'(x)-u'(y)=c(y,x)+\alpha [0]$. Since this equality implies the
 following ones (cf. the proof of \ref{I}) for all integers $n$
\begin{eqnarray*} 
&&\T{n}(u)(x)-\T{n}u(y)=c(y,x)+\alpha [0],\\
&&\TT{n}(u)(x)-\TT{n}u(y)=c(y,x)+\alpha [0].
\end{eqnarray*} 
By domination of $u$, we therefore have for every $n$
\begin{eqnarray}
\T{(n+1)}u(x)+\alpha [0]&=&\T{n}u(y)+c(y,x)+\alpha[0]\nonumber \\
&=&\T{n}u(x),\label{eq1comp}
\end{eqnarray}
and
\begin{eqnarray}
\TT{(n+1)}u(y)-\alpha [0]&=&\TT{n}u(x)-c(y,x)-\alpha[0]\nonumber \\
&=&\TT{n}u(y).\label{eq2comp}
\end{eqnarray}
Using the same argument as in the previous lemma (\ref{IIII}), by successive applications of \ref{atteint} we
 can find  chains $(x_{-n}^n,\ldots
 x_{-1}^n=y,x_0^n=x)$ such that
 $$\forall n \in \mathbb{N},\T{n}u(x)= u(x_{-n}^n)+\sum_{i=-n}^{-1}c(x_i^n,x^n_{i+1}),$$
 and chains $( x_{-1}^n=y,x_0^n=x,\ldots,x_n^n)$ such that
 $$\forall n \in \mathbb{N},\TT{n}u(x)= u(x_{n}^n)-\sum_{i=0}^{n-1}c(x^n_{i},x^n_{i+1}).$$
 Using \ref{eq1comp} and \ref{eq2comp}, we get that 
 $$\forall n \in \mathbb{N},u(x)-u(x_{-n}^n)= \sum_{i=-n}^{-1}c(x^n_i,x^n_{i+1})+n\alpha[0],$$
 and
 $$\forall n \in \mathbb{N},u(x_{n}^n)-u(x)= \sum_{i=0}^{n-1}c(x^n_{i},x^n_{i+1})+n\alpha[0].$$
 Summing these two last equalities we get
 $$\forall n \in \mathbb{N},u(x_{n}^n)-u(x_{-n}^n)= \sum_{i=-n}^{n-1}c(x^n_{i},x^n_{i+1})+2n\alpha[0],$$
which proves that the chains $(x_{-n}^n,\ldots
 x_{-1}^n=y,x_0^n=x,x_1^n,\ldots,x_n^n)$ are calibrating for $u$.
 \\
  By \ref{apriori}, for every integer $k\in \mathbb{Z}$, the sequence $(x_k^n),n\geqslant |k|$ is bounded hence, by a diagonal extraction ($n_l\rightarrow +\infty$ as $l\rightarrow +\infty$)  we can assume each $(x_k^{n_l}),n_l\geqslant |k|$ converges to a $x_k\in X$.
  Let us now fix two integers $m$ and $m'$ such that $m\leqslant m'$. If $n_l\geqslant |m|,|m'|$ we have
  $$u(x_{m'}^{n_l})-u(x_{m}^{n_l})= \sum_{i=m}^{m'-1}c(x_{i}^{n_l},x^{n_l}_{i+1})+(m'-m)\alpha[0],$$
  letting $n$ go to $+\infty$, using the continuity of $u$, the following holds
    $$u(x_{m'})-u(x_{m})= \sum_{i=m}^{m'-1}c(x_{i},x_{i+1})+(m'-m)\alpha[0].$$
    Since $m$ and $m'$ were taken arbitrarily, this proves that
   the sequence $(x_k)_{k\in \mathbb{Z}}$ is calibrating for $u$ and therefore  that $(x,y)\in \widehat{\A}_u$.
Therefore, $u'$ is a sub-solution strict at $X^2
-\widehat{\A}_u$. Moreover, by \ref{IIII} and since  $\sum a_n+\sum b_n =1$, $u$ and $u'$ coincide on $\A$ which finishes to prove the first part of the theorem.
 
 To prove the second part, pick $u$ such that $\A_u=\A$ which is possible
according to \ref{I}. The function $u'$ is strict outside of $\widehat{\A}$.
\end{proof}

\section{Towards a discrete analog of Ma\~ne's potential}
In the study of globally minimizing curves in Lagrangian dynamics, two functions appear naturally. The first one is used to study infinite orbits of the Euler-Lagrange flow and is Mather's Peierls' barrier which was introduced in the Lagrangian setting in \cite{Mat}. This barrier was studied in the discrete case in \cite{Be}. The other function is Ma\~ne's potential and was introduced in \cite{Man}. As it is proved in \cite{FaSi}, Ma\~ne's potential gives nice characterizations of the projected Aubry set in terms of differentiability and weak KAM solutions (see Theorems 4.3 and 5.3 in \cite{FaSi}). However, in the discrete setting, this notion seems less natural.

In this section, we propose two versions of Ma\~ne's potential. It appears that they are closely related. Moreover, by analogy with Fathi and Siconolfi's results, we characterize the Aubry set in terms of continuity of the potential. In order to stay consistent with the rest of the text, we will only consider the critical case. However, all the results of this section hold in the super-critical case (that is to consider the cost $c+\alpha$, $\alpha>\alpha[0]$). Moreover, in this section, let us stress the fact that $X$ and $c$ only need to satisfy the hypothesis of the beginning of the article being that $X$ is a  $B$-length space at scale $K$ for some constants $B$ and $K$ (see \ref{scale} for the exact definition) with compact closed balls and $c$ is continuous, super-linear and uniformly bounded (see conditions \ref{unif} and \ref{unifb} in the introduction).

 The following construction is inspired from Perron's method to
 construct viscosity solutions in PDE. It is also reminiscent of ideas of Gabriel Paternain and results obtained in \cite{FaSi}.
\begin{df}\rm 
 We define the potential 
$$\varphi (x,y)=\sup_{u\<c+\alpha [0]} u(y)-u(x),$$ 
where the supremum is taken over all critical sub-solutions (not necessarily continuous).
\end{df}
We begin with some properties.
\begin{pr}\label{pprevious}
 The potential satisfies the following properties:
 
 \begin{enumerate}
\item For all $(x,y)\in X^2$ we have $\varphi(x,y)\leqslant c(x,y)+\alpha [0]$. In particular, the potential is everywhere finite.
\item For all $x\in X$, the potential verifies $\varphi(x,x)=0$.
\item A function $u$ is critically dominated if and only if for all
  $(x,y)$ in $X^2$ we have $u(y)-u(x)\leqslant \varphi (x,y)$.
\item The function $\varphi$ verifies the triangular inequality that is for all $x,y,z$ in $X$ we have $\varphi(x,y)+\varphi(y,z)\geqslant \varphi(x,z)$.
\end{enumerate}
\end{pr}
In particular, this proves points (1) and (2) of theorem \ref{Mane}.
\begin{proof}

 Items 1. and 2. are clear. The third one comes from the
 fact that for any dominated function $u$ we clearly have that
 $$\forall (x,y)\in X^2,u(y)-u(x)\leqslant \varphi(x,y).$$
 For the reverse implication, since by the first point of the proposition, we have $\varphi(x,y)\leqslant c(x,y)+\alpha[0]$, any function which satisfies 
$$\forall(x,y)\in X^2,u(y)-u(x)\leqslant \varphi(x,y),$$
is necessarily critically dominated.
 The fourth point is clear
 from the definition. 
\end{proof}

Before going any further, let us state two simple lemmas that we will use throughout this section. The first one helps to understand how
to construct sub-solutions: 
\begin{lm}\label{entre}
 Let $u\<c+\alpha [0]$ and let a function $v$ that verifies the
 following inequalities
$$u\leqslant v\leqslant \ttt u+\alpha [0].$$
Then $v$ itself is a critical sub-solution: $v\<c+\alpha [0]$.
\end{lm}
\begin{proof}
 The proof is merely based on the monotonicity of the Lax-Oleinik semi-group
$$u\leqslant v\leqslant \ttt u +\alpha [0]\leqslant \ttt v+\alpha [0],$$
which proves that $v$ is itself critically dominated.
\end{proof}

\begin{lm}
 Let $u$ be any critical sub-solution and $x\in \A_u$, then $u$ is continuous at $x$.
\end{lm}
\begin{proof}
 The following inequalities are true
$$\TTT u-\alpha[0]\leqslant u\leqslant \ttt u+\alpha[0]$$
and are equalities at $x$. Therefore, the conclusion is a direct consequence of the fact that both $\ttt u+\alpha[0]$ and $\TTT u-\alpha[0]$ are continuous (cf. \ref{HH}).
\end{proof}

The reason why we are interested in this potential is that it
generates the greatest possible sub-solutions.
\begin{pr}\label{previous}
The potential verifies the following properties: 
 \begin{enumerate}
 \item for all $x\in X$, the function $\varphi_x=\varphi(x,.)$ is a critical sub-solution.
\item Let $x\in X$, then for any $y\neq x$ we have
$$\varphi_x(y)=\ttt \varphi_x(y)+\alpha [0],$$ 
therefore, the function $\varphi_x$ is lower
  semi-continuous, and continuous on $X\setminus \{x\}$.
\item A point $x\in X$ is in the projected Aubry set if and only if the function $\varphi_x$ is a weak KAM solution.
\item If the point $x\in X$ is not isolated, the function $\varphi_x$ is continuous at $x$ if and only if $x\in \A$. 
 \end{enumerate}
\end{pr}
In particular, this ends the proof of theorem \ref{Mane}.
\begin{proof}
 The first part is a direct consequence of part 4 and part 3 of the
 previous proposition (\ref{pprevious}).

 Let us consider the function $\psi_x$ defined
 as follows:
\begin{itemize} 
\item $\psi_x(x)=\varphi_x(x)=0$,
 \item $\psi_x(y)=\ttt \varphi_x(y)+\alpha
 [0]$ if $y\neq x$.
\end{itemize}
 The function $\psi_x$ is lower
 semi-continuous. As a matter of fact, it is continuous outside of $x$
 and at $x$ it verifies
$$\liminf_{y\rightarrow x}\psi_x(y)=\liminf_{y\rightarrow x}\ttt \varphi_x(y)+\alpha [0]\geqslant \liminf_{y\rightarrow x}\varphi_x(y)\geqslant 0=\psi_x(x),$$
where the last inequality follows from the existence of a continuous critical sub-solution $u$ which implies
$$\liminf_{y\rightarrow x}\varphi_x(y)\geqslant \liminf_{y\rightarrow
  x}u(y)-u(x)=0.$$ \\
Note that $\varphi_x\leqslant \psi_x\leqslant \ttt\varphi_x+\alpha[0]$ therefore using the ``in between'' lemma (\ref{entre}), we
obtain at once that the function $\psi_x$ is critically dominated and
greater or equal to $\varphi_x$ by definition. Since, by definition of $\varphi$ we also have
$$\forall y\in X, \varphi_x(y)\geqslant \psi_x(y)-\psi_x(x)=\psi_x(y),$$
we obtain in fact that $\varphi_x=\psi_x$. In particular, $\varphi_x=\ttt \varphi_x+\alpha[0]$ on $X\setminus \{x\}$.  This finishes the proof of point (2).\\

To prove 3, note that if $x\in \A$, then for any sub-solution $u$, the following equality holds by \ref{IIII}
$$\ttt u(x)+\alpha[0]=u(x).$$
In particular, $\varphi_x(x)=\ttt \varphi_x(x)+\alpha[0]$, and by the previous point, those functions also coincide on $X\setminus \{x\}$.
\\
To prove the converse, assume $x\notin \A$ and pick a sub-solution $u$ which is strict at $x$ (such a function exists by \ref{SStrict}). Without loss of generality, we can assume that $u(x)=0$.
In particular, the following holds
$$\ttt u(x)+\alpha[0]>u(x)=0.$$
We already
 know that 
$$\forall y\in X, u(y)\leqslant \varphi_x(y).$$
 By the
 monotonicity of the Lax-Oleinik semi-group, we obtain that
$$\forall y\in X, \ttt u(y)+\alpha[0]\leqslant \ttt
\varphi_x(y)+\alpha [0].$$
 Taking $y=x$, we obtain that
$$\varphi_x(x)=0<\ttt u(x)+\alpha [0] \leqslant \ttt
\varphi_x(y)+\alpha [0].$$

Finally, let us assume $x\in X$ is not isolated. We prove that $\varphi_x$ is continuous at $x$ if and
only if $x\in\A$. Assume first that $x\notin \A$. Pick $u\<c+\alpha
[0]$ such that $u$ is strict at $x$ and that $u$ is continuous and
vanishes at $x$. We can find an open neighborhood $V$ of $x$ and an
$\e >0$ such that on $V$, $u+\e \leqslant \ttt u+\alpha [0]$ and $|u|\leqslant \frac{\e}{2}$. Now the
function $v=u+\e \chi_{V\setminus\{x\}}$  verifies $v(x)=0$. Again it is dominated by \ref{entre}. Therefore we have that if $y\in
V\setminus\{x\}$ (which is not empty because $x$ is not isolated), 
$$\varphi_x(y)\geqslant v(y)=u(y)+\e\geqslant \frac{\e}{2},$$
which proves that $\varphi_x$ is not continuous at $x$. The other
implication is clear since we know that any sub-solution is continuous
at $x$ as soon as $x\in \A$.
\end{proof}
Part 2 of proposition \ref{previous} shows that when $x\notin \A$, the
function $\varphi_x$ has a lower jump at $x$. Here is a property of this
``jump''. It is a direct consequence of the previous proposition:

\begin{lm}
For any $x\in X$, the quantity $F(x)=\sup_{u\<c+\alpha [0]} \ttt u(x)+\alpha [0] -u(x)$, where this supremum is taken on the set of all sub-solutions, exists and is equal to $\ttt \varphi_x(x)+\alpha
[0]$. \\ Moreover, for any non isolated point $x$, the function $F$ verifies
$$F(x)=\lim_{\substack{y\rightarrow
  x\\y\neq x}}\varphi_x(y).$$
\end{lm}
\begin{proof}
 For the first equality,  let $u$
 be any critically dominated function and let $x\in X$. We already
 know that 
$$\forall y\in X, u(y)-u(x)\leqslant \varphi_x(y).$$
 By the
 monotonicity of the Lax-Oleinik semi-group, we obtain that
$$\forall y\in X, \ttt u(y) -u(x)+\alpha[0]\leqslant \ttt
\varphi_x(y)+\alpha [0].$$
 Taking $y=x$, we obtain that
$$\ttt u(x)+\alpha [0] -u(x)\leqslant \ttt \varphi_x(x)+\alpha[0].$$
Therefore, the supremum in the definition of $F(x)$ is reached by the sub-solution $\varphi_x$:
$$F(x)=\ttt \varphi_x(x)+\alpha
[0]-\varphi_x(x),$$
since $\varphi_x(x)=0$.\\
 Now,
The continuity of the function $\ttt\varphi_x+\alpha[0]$ at $x$ together with the equality $\varphi_x=\ttt\varphi_x+\alpha[0]$ on $X\setminus \{x\}$ imply the second equality. 
\end{proof}

Let us now ``reverse time'' and look what happens when we consider the
reversed Lax-Oleinik semi-group:
$$\TTT u(x)=\sup_{y\in X} u(y)-c(x,y).$$
This semi group may also be interpreted as a negative Lax-Oleinik semi-group for the symmetric cost $\overline{c}(x,y)=c(y,x)$ by the following relation:
$$\TTT u=-T_{\overline{c}}^-(-u).$$

Let us stress the fact that the critical value is unchanged when considering the positive semi-group $\TTT$. As a matter of fact, the critical value is the smallest $\alpha$ such that there exists $u\<c+\alpha$. But $u\<c+\alpha$ if and only if $-u\<\overline{c}+\alpha$. Hence the critical values are the same.\\
Therefore, the same properties, with the same proofs, hold. Let us simply state the results.

\begin{lm}\label{entre+}
 Let $u\<c+\alpha [0]$ and let $v$ be a function that verifies the
 following inequalities:
$$u\geqslant v\geqslant \TTT u-\alpha [0],$$
then $v$ itself is a sub-solution: $v\<c+\alpha [0]$.
\end{lm}
\begin{pr} 
The function $\varphi$ verifies the following properties:
 \begin{enumerate}
 \item for all $x\in X$, the function $\varphi^x=-\varphi(.,x)$ is a
   critical sub-solution.
\item Let $x\in X$, then for any $y\neq x$ the function $\varphi^x$ verifies
$$\varphi^x(y)=\TTT \varphi^x(y)-\alpha [0]$$ therefore, it is upper
  semi-continuous, and continuous on $X\setminus \{x\}$.
\item A point $x\in X$ is in the projected Aubry set if and only if the function $\varphi^x$ is a positive weak KAM solution.
\item If $x$ is not isolated, the function $\varphi^x$ is continuous at $x$ if and only if $x\in \A$.
 \end{enumerate}
\end{pr}
\begin{lm}
For any $x\in X$, the quantity $f(x)=\inf_{u\<c+\alpha [0]} \TTT u(x)-\alpha [0]-u(x)$ exists and is equal to $\TTT \varphi^x(x)-\alpha
[0]$. \\ Moreover, whenever $x$ is not isolated, the function $f$ verifies
$$\forall x\in X,f(x)=\lim_{\substack{y\rightarrow
  x\\y\neq x}}\varphi^x(y). $$
\end{lm}

 Until now, we mostly considered
general sub-solutions. However, it is much easier to deal with
semi-continuous or even continuous functions. We have already noticed
that the functions $\varphi_x$ are lower semi-continuous and therefore
that in the definition of $\varphi$ we can restrict the supremum to
lower semi-continuous functions. The following theorem strengthens the
result.

\begin{Th}
 Let $x\in X$.  The function $\varphi_x$ is a simple limit of
 continuous critical sub-solutions. Moreover, the limit may be chosen
 to be uniform outside of any given neighborhood of $x$.
\end{Th}
\begin{proof}
If $x\in \A$, the function $\varphi_x$ is a weak KAM solution and is
therefore continuous.  If $x\notin \A$, then $\ttt \varphi_x(x)+\alpha[0]>0$.
Let $\e\in ]0,1[$ be such that $\e<\ttt \varphi_x(x)+\alpha[0]$. We will see in the appendix (\ref{HH} and \ref{apriori})  that any sub-solution has a growth that is at most
        linear (and which can be bounded independently from the sub-solution) while $c$ is super-linear. Therefore, we can find a
        real number $1<R$ such that whenever $y\in B(x,1)$ and
        $\d(x,z)>R$ then for any critical sub-solution $u$,
\begin{equation}\label{111}
u(y)-u(z)<c(z,y)+\alpha[0]-2(\ttt \varphi_x(x)+\alpha [0])
\end{equation}
 and
\begin{equation}\label{22}
u(z)-u(y)<c(y,z)+\alpha[0]-2(\ttt \varphi_x(x)+\alpha [0]).
\end{equation}
 Using the continuity of $c$ and the compactness of the ball $B(x,R)$, we can find a neighborhood $V\subset B(x,1)$  of $x$
        verifying: 
\begin{itemize}
\item if $y,z,t,u\in V$ then
        $|c(y,z)-c(t,u)|<\frac{\e}{2}$,
\item if $z\in B(x,R)$ and
        $y,t\in V$ then $$|c(z,y)-c(z,t)|<\e$$ and
        $$|c(y,z)-c(t,z)|<\e,$$
\end{itemize}
 Cutting down $V$, by continuity of $\ttt \varphi_x$ we can assume

\begin{itemize}
\item if $y\in V\setminus \{x\}$ then
        $\varphi_x(y)=\ttt \varphi_x(y)+\alpha [0]>\e$,
\item if $y,t\in
        V$ then $|\ttt
        \varphi_x(y)-\ttt \varphi_x(t)|<\frac{\e}{2}$,
\end{itemize}
Note that from the last condition it follows that for $(y,t)\in V\setminus \{x\}$ we have 
$$|\varphi_x(y)- \varphi_x(t)|=|\ttt
        \varphi_x(y)-\ttt \varphi_x(t)|<\frac{\e}{2}.$$
 Let us now consider the function $\fe$ defined as
        follows. 
Let $\theta:X\rightarrow [0,1]$ be a Urysohn function equal to $1$ on $X\setminus V$, which vanishes at $x$ and  define
$$\forall z\in X, \fe(z)=\theta(z)\left(\ttt \varphi_x(z)-\e\right)=\theta(z)\left( \varphi_x(z)-\e\right).$$
The function $\varphi_{\e}$  verifies the following properties:
\begin{itemize} 
 \item on $X\setminus V$,
        $\fe(y)=\varphi_x(y)-\e$,
\item on $V$, $\fe$ is non-negative, vanishes at $x$ and
        verifies
$$\forall y\in V\setminus\{x\}, \fe(y)\leqslant \varphi_x(y)-\e.$$
\end{itemize}
  Now let
        us check that the function $\fe$ is critically dominated. It
        is enough to separately consider several cases.
 If both
        $y,z\notin V$ then
        $$\fe(y)-\fe(z)=\varphi_x(y)-\varphi_x(z)\leqslant
        c(z,y)+\alpha [0].$$
If $y\in V$ and $z\notin V$, 
        we distinguish between cases. First, let us notice that if
        $z\notin B(x,R)$ then, since $\fe$ is non negative on $V$, 
taking into consideration the fact that 
$$\ttt \varphi_x(x)-\varphi_x(y)+\alpha [0]+\frac{\e}{2}\geqslant 0,$$
which is clear for $y=x$, since $\ttt \varphi_x(x)+\alpha[0]\geqslant \varphi_x(x)=0$, and for $y\neq x$, follows from 
$$|\ttt \varphi_x(x)-\ttt \varphi_x(y)|<\frac{\e}{2} \; \textrm{and} \; 
\varphi_x(y)=\ttt \varphi_x(y)+\alpha[0],$$ 
and the fact (using \eqref{22}) that 
$$\varphi_x(z)-\varphi_x(y)\leqslant c(y,z)+\alpha[0]-2(\ttt \varphi_x(x)+\alpha[0]),$$
we obtain that
\begin{eqnarray*}
\fe(z)-\fe(y)&\leqslant&\varphi_x(z)-\e\\
&\leqslant&  \varphi_x(z)-\e -\varphi_x(y)+\ttt  \varphi_x(x)+\alpha [0]+\frac{\e}{2}\\
&\leqslant& c(y,z)+\alpha[0]-2(\ttt \varphi_x(x)+\alpha[0])+\ttt  \varphi_x(x)+\alpha [0]-\frac{\e}{2}\\
&<&c(y,z)+\alpha [0],
\end{eqnarray*}
because $\ttt \varphi_x(x)+\alpha[0]\geqslant \varphi_x(x)=0$.
\\ If $z\in B(x,R)$ then using $\varphi_x(x)=0$ and $\fe(y)\geqslant0$, we obtain
$$\fe(z)-\fe(y)\leqslant \varphi_x(z)-\e-\varphi_x(x)\leqslant c(x,z)+\alpha [0]-\e\leqslant c(y,z)+\alpha [0].$$
In both cases, the following inequalities hold
$$\fe(y)-\fe(z)\leqslant \varphi_x(y)-\e-(\varphi_x(z)-\e)=\varphi_x(y)-\varphi_x(z)\leqslant
c(z,y)+\alpha [0].$$ 
Finally, if $y,z\in V$ then since $\varphi_x(x)=0$ and $\fe(z)\geqslant 0$, 
$$\fe(y)-\fe(z)\leqslant  \varphi_x(y)-\varphi_x(x)-\e\leqslant
c(x,y)+\alpha [0]-\e\leqslant c(z,y)+\alpha [0].$$ 
\end{proof}
We now propose another version of a discrete Ma\~ ne's potential. We will show that it is very much related to $\varphi$. 
We begin with a definition 
\begin{df}\rm
Let us define the family of functions, for all $n\in \mathbb{N}^*,(x,y)\in X^2$,
$$c_n(x,y)=\inf_{(x_1,\ldots ,x_{n-1})\in X^{n-1}}\{c(x,x_1)+c(x_1,x_2)+\cdots +c(x_{n-1},y) \}.$$
\end{df}
\begin{pr}
For any $n>0$, the function $c_n$ is continuous.
\end{pr}
\begin{proof}
Let $n$ be a positive integer and let us consider  a pair of points $(x^0,y^0)\in X^2$.
First, let us notice that for all $(x,y)\in K=\overline{B(x^0,1)\times B(y^0,1)}$, using the uniform boundedness of $c$ (condition \ref{unifb}), the following inequality holds:
\begin{equation}\label{fini}
c_n(x,y)\leqslant (n-1)c(x,x)+c(x,y)\leqslant  nA(\d(x^0,y^0)+2).
\end{equation}
Moreover, using the super-linearity of $c$ (condition \ref{unif}), for any chain of points $(x_1,\ldots ,x_{n-1})\in X^{n-1}$, we have, setting $x_0=x$ and $x_n=y$:
\begin{equation}\label{fini2}
\sum_{i=0}^{n-1} c(x_i,x_{i+1}) \geqslant -nC(1)+\sum_{i=0}^{n-1} \d(x_i,x_{i+1}).
\end{equation} 
Finally, if the chain verifies that $\sum_{i=0}^{n-1} c(x_i,x_{i+1})\leqslant c_n(x,y)+1$, using \eqref{fini} and \eqref{fini2}, we obtain that 
$$\sum_{i=0}^{n-1} \d(x_i,x_{i+1})\leqslant c_n(x,y)+nC(1)+1\leqslant n(A(\d(x^0,y^0)+2)+C(1))+1=R.$$
In particular,
 \begin{eqnarray*}
 \forall i\in [0,n], \d(x_0,x_{i})& \leqslant& \sum_{j=0}^{i-1} \d(x_j,x_{j+1})\\
 & \leqslant &\sum_{j=0}^{n-1} \d(x_j,x_{j+1})\\
 &\leqslant& R.
 \end{eqnarray*}
  We have just proven that restricted to $K$, in the definition of $c_n$, we can take the infimum on chains of points which belong to $B(x,R)^{n-1}$ which is relatively compact. Therefore, by Heine's theorem, the restriction of $c_n$ to $K$ is a finite infimum of equi-continuous functions and is therefore itself continuous.
  
\end{proof}
\begin{rem}\rm
In the case where $X$ is compact, one can show that the family of functions $(c_n)_{n\in \mathbb{N}^*}$  is uniformly equi-continuous, however, in the non compact case, it is not clear whether this fact remains true.
\end{rem}
Let us now introduce another family of functions: 
\begin{df}\rm
For any $n\in \mathbb{N}^*$ and $(x,y)\in X^2$ let 
$$\varphi_n(x,y)=\inf_{k\geqslant n}c_k(x,y)+k\alpha [0].$$
\end{df}
This quantity is always greater or equal to $\varphi(x,y)$ by the triangular inequality. Moreover, the functions $\varphi_n$ are clearly increasing with $n$. 

\begin{pr}\label{semi}
 For any $n\in \mathbb{N}^*$, the function $\varphi_n$ is upper semi-continuous. Moreover, for any $x$, the function $\varphi_{n,x}=\varphi_n(x,.)$ is critically dominated. Finally, $\ttt \varphi_{n,x}+\alpha [0]=\varphi_{n+1,x}$.
\end{pr}
\begin{proof}
The upper semi-continuity comes from the fact that $\varphi_n$ is an infimum of continuous functions.
 The domination of $\varphi_{n,x}$ is consequence of the definitions. In fact, let $y,z$ be
 in $X$, then
\begin{eqnarray*}
\varphi_{n,x}(y)+c(y,z)+\alpha [0]&=&\inf_{k\geqslant n}c_k(x,y)+k\alpha [0]+c(y,z)+\alpha [0]
\\ &\geqslant &\inf_{k\geqslant n+1}c_k(x,z)+k\alpha [0]
\\&=&\varphi_{n+1,x}(z)\\
&\geqslant &\varphi_{n,x}(z).
\end{eqnarray*}

To prove the last point, just write that 
\begin{eqnarray*}
\ttt \varphi_{n,x}(z)+\alpha [0]&=&\inf_y \varphi_{n,x}(y)+c(y,z)+\alpha [0]
\\&=&\inf_{y\in X} \inf_{k\geqslant n}c_k(x,y)+k\alpha [0]+c(y,z)+\alpha [0]
\\&=&\varphi_{n+1,x}(z) .
\end{eqnarray*}

\end{proof}
We now link both versions of the potential:
\begin{pr}\label{identification}
On $X^2\setminus \Delta X$, $\varphi=\varphi_1$.  Moreover, for any $x\in X$, 
$$\varphi_1(x,x)\geqslant \varphi(x,x)=0.$$
\end{pr}
\begin{proof}
By definition of $\varphi_{1,x}$, if $u\<c+\alpha[0]$,
$$\forall y\in X,u(y)-u(x)\leqslant \varphi_{1,x}(y),$$
therefore, $\varphi_x\leqslant \varphi_{1,x}$.
 \\
We then notice that $\TTT \varphi_{1,x}(x)-\alpha[0]\leqslant 0$. As a matter of fact, for any $x_1\in X$ we have
$$\varphi_{1,x}(x_1)-c(x,x_1)-\alpha[0]\leqslant 0,$$
by definition of the function $\varphi_1$. Taking the supremum on $x_1$, we get the result.

Let us define the function 
$\psi$ by 
\begin{itemize}
 \item $\psi(y)=\varphi_{1,x}(y)$ if $y\neq x$,
\item $\psi (x)=0$.
\end{itemize}
Since $\varphi_{1,x}\geqslant \psi \geqslant \TTT \varphi_{1,x}-\alpha[0]$ the ``in-between'' lemma (\ref{entre+}) gives that the function $\psi$ is a critical sub-solution. But $\psi$ vanishes at $x$ and is greater than $\varphi_x$, therefore $\psi=\varphi_x$.
\end{proof}
As a corollary of the previous proof we also obtain the following:
\begin{co}\label{vanish}
The following equality holds: 
$$\forall x\in X,\TTT \varphi_{1,x}(x)-\alpha[0]= 0.$$
\end{co}
\begin{proof}
Let us fix an $x\in X$. We just saw that  $\TTT \varphi_{1,x}(x)-\alpha[0]\leqslant 0$. Assume now by contradiction that we can find an $\e >0$ such that
$$\TTT \varphi_{1,x}(x)-\alpha[0]\leqslant -\e<0=\varphi_x(x)\leqslant \varphi_{1,x}(x).$$
By analogy with the previous proof, let us define the function
$\psi$ by 
\begin{itemize}
 \item $\psi(y)=\varphi_{1,x}(y)$ if $y\neq x$,
\item $\psi (x)=-\e$
\end{itemize}
Since $\varphi_{1,x}\geqslant \psi \geqslant \TTT \varphi_{1,x}-\alpha[0]$ the ``in between'' lemma (\ref{entre+}) gives that the function $\psi$ is a critical sub-solution. But if $y\neq x$ then $\psi(y)-\psi(x)>\varphi_x(y)$ which is in contradiction with the definition of $\varphi$.
\end{proof}
In the following, we will use this lemma:
 \begin{lm}\label{LO}
 Let $u:X\rightarrow \R{}$ be a function and $n\in \mathbb{N}$, then
 $$\T{n}\TT{n}u\geqslant u\mathrm{\ \ and\ \  }\TT{n}\T{n}u\leqslant u.$$
  Moreover, if
 $u$ is a negative (resp. positive) weak KAM solution then
 $$\T{n}\TT{n}u= u\ \ (\mathrm{resp.}\ \   \TT{n}\T{n}u= u).$$
  Finally, the operators
 $\ttt\circ \TTT$ and $\ttt\circ \TTT$ are idempotent.
 \end{lm}
 \begin{proof}
 By symmetry, we will only prove one half of the lemma. By definition,
 for a given $x\in X$ we have
 $$\ttt\TTT u(x)=\inf_z \sup_y u(y)-c(z,y)+c(z,x),$$
  and this quantity
 is greater than $u(x)$ (take $y=x$). Now the first part of the
 proposition is obtained by induction or by applying the argument to $c_n$ instead of $c$.\\
  If $u$ is a negative
 weak KAM solution, we have that $u\geqslant \TTT u-\alpha[0]$ (this is
 always true for a dominated function) and therefore 
 $$u=\ttt u+\alpha[0]\geqslant \ttt \TTT u.$$
  Hence we have in fact an
 equality. Once again, the general result follows by induction or by using $c_n$ instead of $c$.
 
  Finally, we have already seen that
 $(\ttt\circ \TTT)^2\geqslant \ttt\circ \TTT$. For the reversed
 inequality,
 note that since similarly, $\TTT\circ \ttt\leqslant Id$,
 $$\ttt\circ (\TTT\circ \ttt)\circ \TTT\leqslant \ttt\circ \TTT.$$
 \end{proof}
\begin{pr}\label{equ}
Let $x\in X$ be any point, then the following inequality holds: $\varphi_{1,x}(x)\leqslant \ttt \varphi_x(x)+\alpha[0]$. In particular, the function $\varphi_{1,x}$ is continuous. Moreover, if the point $x$ is not isolated, we have in fact an equality: $\varphi_{1,x}(x)= \ttt \varphi_x(x)+\alpha[0]$.
\end{pr}
\begin{proof}
We have already seen (\ref{vanish}) that $\TTT \varphi_{1,x}(x)-\alpha[0]=0$. Therefore, the following inequality is true:
$$\TTT \varphi_{1,x}-\alpha[0]\leqslant \varphi_{x}.$$
As a matter of fact, it is true at $x$, and at other points $y$, it is a consequence of the equality $\varphi_{1,x}(y)=\varphi_x(y)$ (\ref{identification}) and of the fact that since $ \varphi_{1,x}$ is a critical sub-solution, we have $\TTT \varphi_{1,x}-\alpha[0]\leqslant \varphi_{1,x}$. By the monotonicity of the Lax-Oleinik semi-group, the following holds 
$$\ttt\TTT \varphi_{1,x}\leqslant \ttt\varphi_{x}+\alpha[0],$$
which by (\ref{LO}) gives us
$$\varphi_{1,x}\leqslant \ttt\TTT \varphi_{1,x}\leqslant \ttt\varphi_{x}+\alpha[0].$$
By (\ref{previous}) and (\ref{identification}) these inequalities are in fact equalities, except possibly at $x$. Since by (\ref{HH}) the function $\ttt\varphi_{x}+\alpha[0]$ is continuous it is clear that $\varphi_{1,x}$ is lower semi-continuous and therefore continuous by (\ref{semi}). \\
Finally, the equality $\varphi_{1,x}(x)= \ttt \varphi_x(x)+\alpha[0]$ whenever $x$ is not isolated is a straight consequence of the continuity of the functions $\varphi_{1,x}$ and  $\ttt \varphi_x+\alpha[0]$ and of the fact that they coincide on $X\setminus \{x\}$.
\end{proof}
Actually, the last equality of the previous proposition (\ref{equ}) holds even when $x$ is isolated, as shown below:
\begin{pr}\label{egalitetotale}
For any $x\in X$, the following holds
$$\forall y\in X,\varphi_{1,x}(y)= \ttt \varphi_x(y)+\alpha[0].$$
\end{pr}
\begin{proof}
We have already proven the result when $y\neq x$ and we proved above (\ref{equ}) that
$$\varphi_{1,x}(x)\leqslant \ttt \varphi_x(x)+\alpha[0].$$
Let us prove the reverse inequality. By definition and monotonicity of the Lax-Oleinik semi-group, since $\varphi_{1,x}\geqslant \varphi_x$ the following holds
\begin{eqnarray*}
\forall x\in X, \ttt \varphi_x(x)+\alpha[0]&=&\inf_{y\in X}\varphi_x(y)+c(y,x)+\alpha[0]\\
&\leqslant &\inf_{y\in X}\varphi_{1,x}(y)+c(y,x)+\alpha[0]\\
&=&\ttt \varphi_{1,x}(x)+\alpha[0]=\varphi_{2,x}(x),
\end{eqnarray*}
where we used the last part of \ref{semi} for the last equality. Taking $y=x$ in infimum of the Lax-Oleinik we also have
$$\ttt \varphi_x(x)+\alpha[0]\leqslant c(x,x)+\alpha[0].$$
Since $\varphi_{1,x}(x)=\min (c(x,x)+\alpha[0],\varphi_{2,x}(x))$, this finishes the proof of the proposition.
\end{proof}
Obviously, similar results hold when considering the positive time Lax-Oleinik semi-group $\TTT$ therefore, we obtain the following:

\begin{pr}\label{semi+}
 For any $n\in \mathbb{N}^*$, and any $x$, the function $\varphi^{n,x}=-\varphi_n(.,x)$ is critically dominated. Finally, $\TTT \varphi^{n,x}-\alpha [0]=\varphi^{n+1,x}$.
\end{pr}
\begin{lm}\label{vanish+}
The following equality holds: 
$$\forall x\in X,\ttt \varphi^{1,x}(x)+\alpha[0]= 0.$$
\end{lm}
\begin{pr}\label{equ+}
Let $x\in X$ be any point, then the following equality holds: $\varphi^{1,x}(x)= \TTT \varphi^x(x)-\alpha[0]$. In particular, the function $\varphi^{1,x}$ is continuous. 
\end{pr}
We are now able to prove the following theorem:
\begin{Th}\label{conti}
  The family of  functions $\varphi_n, n\in \mathbb{N}$ is locally equi-continuous on $X^2$. In particular, $\varphi_1$ is a continuous extension of $\varphi$ restricted to $X^2\setminus \Delta X$.
\end{Th}
\begin{proof}
We first prove the continuity of $\varphi_1$. Let $(x,y)\in X^2$
By (\ref{HH}) we know that images of critically dominated functions by the Lax-Oleinik semi-groups are locally equi-continuous. Therefore, let us consider relatively compact neighborhoods $V$ and $V'$ of respectively $x$ and $y$ and let $\om$ be a modulus of continuity for  images of critically  dominated functions by the Lax-Oleinik semi-groups restricted to $V$ and $V'$. Let now $(x',y')\in V\times V'$. Using (\ref{egalitetotale}) and (\ref{equ+}) we obtain
\begin{eqnarray*}
|\varphi_1(x,y)-\varphi_1(x',y')|&\leqslant &|\varphi_1(x,y)-\varphi_1(x,y')|+|\varphi_1(x,y')-\varphi_1(x',y')|\\
&\leqslant&|\ttt\varphi_{x}(y)-\ttt\varphi_{x}(y')|+|\TTT\varphi^{y'}(x)-\TTT\varphi{y'}(x')|\\
&\leqslant&\om(\d(y,y'))+\om(\d(x,x')).
\end{eqnarray*}
This proves the continuity of $\varphi_1$. Similarly, if $n\geqslant 2$ we have 
\begin{eqnarray*}
|\varphi_n(x,y)-\varphi_n(x',y')|&\leqslant &|\varphi_n(x,y)-\varphi_n(x,y')|+|\varphi_n(x,y')-\varphi_n(x',y')|\\
&\leqslant&|\ttt\varphi_{n-1,x}(y)-\ttt\varphi_{n-1,x}(y')|\\
&+&|\TTT\varphi^{n-1,y'}(x)-\TTT\varphi^{n-1,y'}(x')|\\
&\leqslant&\om(\d(y,y'))+\om(\d(x,x')).
\end{eqnarray*}
This proves the local equi-continuity.
\end{proof}
\begin{rem}
\rm
It is clear that whenever a point $x\in X$ is not isolated, the continuous extension of the potential $\varphi$ is unique at $(x,x)$.
\end{rem}

In what follows, we will need this definition:
\begin{df}\rm
Let us define the Peierls barrier
$$h(x,y)=\liminf_{n\rightarrow +\infty} c_n(x,y)+n\alpha [0]=\lim_{n\rightarrow +\infty}\varphi_n(x,y).$$
\end{df}
\begin{lm}
The following inequality is verified:
 $\varphi \leqslant h$.
 \end{lm}
 \begin{proof}
 This point comes from the fact that by
 definition,
$$h(x,y)=\liminf_{n\rightarrow +\infty} c_n(x,y)+n\alpha [0]$$
 while  by the triangular inequality we have 
$$\varphi(x,y)\leqslant \inf_{n\rightarrow +\infty}
 c_n(x,y)+n\alpha [0].$$
 \end{proof} 
In Mather's original work (\cite{Ma2}) , the projected Aubry set is not defined the way we did,
however, we will now prove that our definition is equivalent to the one using the Peierls barrier.
Note that the Peierls barrier $h$ takes
its values in $\R{}\cup \{+\infty\}$ and that it is continuous
whenever it is finite by equi-continuity of the $\varphi_n$ (\ref{conti}). Furthermore, since the  functions $(\varphi_n)$ are critically dominated, it follows that  family of functions $(\varphi_n)_{n\in \mathbb{N}}$ is equi-Lipschitz in the large (\ref{lip}). Therefore, the Peierls barrier is either finite everywhere or $+\infty$ everywhere.
 First, let us give some properties of $h$ which
are proved in the compact case in \cite{Be} and in the continuous case
in \cite{FaSi}. The proof carries on similarly in the general case with the use of \ref{apriori}:

\begin{pr}\label{hh}
For each $n,m\in \mathbb{N}$, $x,y,z\in X$, we have
$$\varphi_{n+m}(x,z)\leqslant \varphi_n(x,y)+c_m(y,z)+m\alpha [0],$$
$$h(x,z)\leqslant h(x,y)+c_m(y,z)+m\alpha [0],$$
$$h(x,z)\leqslant c_m(x,y)+h(y,z)+m\alpha [0].$$
This gives another proof that the function $h$ is either everywhere finite or identically $+\infty$. Moreover, when $h$ is finite, by \ref{conti}, it is continuous.
\\
For each $l,m,n\in \mathbb{N}$ such that $n\leqslant l+m$, for each $x,y,z\in X$ we have
$$\varphi_n(x,z)\leqslant \varphi_m(x,y)+\varphi_l(y,z),$$
$$h(x,z)\leqslant h(x,y)+\varphi_n(y,z),$$
$$h(x,z)\leqslant h(x,y)+h(y,z).$$
\end{pr}
\begin{Th}\label{Peierl-}
 If $x\in X$, and the Peierls barrier is finite, let us define the
 functions $h_x=h(x,.)$ and $h^x=-h(.,x)$. Then $h^x$, $h_x$ are
 respectively a positive and a negative weak KAM solution. 
\end{Th}
\begin{proof}
  We only prove the theorem for the functions $h_x$, the
   rest is similar. Recall that $h_x$ is the limit of the $\varphi_{n,x}$ and is therefore critically dominated. Moreover, by Dini's theorem, since the sequence of functions $\varphi_{n,x}$ is increasing, its convergence is uniform on compact subsets. Therefore, by the continuity property of $\ttt$ (\ref{HH}) the following holds
 \begin{eqnarray*}
 \ttt h_x+\alpha[0]&=&\ttt\left(\lim_{n\rightarrow +\infty} \varphi_{n,x}+\alpha[0]\right)\\
 &=&\lim_{n\rightarrow +\infty}\ttt \varphi_{n,x}+\alpha[0]
 \\&=&\lim_{n\rightarrow +\infty} \varphi_{n+1,x}+\alpha[0]
 \\&=&h_x.
 \end{eqnarray*}
\end{proof}
\begin{co}
 For each $n\in \mathbb{N}$, $x,y\in X$ we have 
$$h(x,y)=\min_{z\in X}h(x,z)+c_n(z,y)+n\alpha [0]=\min_{z\in X}c_n(x,z)+n\alpha [0]+h(z,y).$$
\end{co}
\begin{proof}
It is a straight consequence of (\ref{Peierl-}) and of point (iv) of (\ref{HH}).
\end{proof}

We will now prove a characterization of the Aubry set:

\begin{Th}\label{A}
 The projected Aubry set $\A$ coincides with the
 set $$\A=\{x,h(x,x)=0\}.$$
\end{Th}
Before proving \ref{A}, we need some results about what happens when $h$ is finite. They are very closely related to results in the compact case.
\begin{Th}\label{inegh}
 Let $u\<c+\alpha [0]$, then for all $n,m\in \mathbb {N}$, and for every $x,y\in X$ we have 
$$\forall (x,y)\in X^2,h(x,y)\geqslant \T{n}u(y)-\TT{m}u(x)+(n+m)\alpha [0].$$
\end{Th}
\begin{proof}
Let $n,m\in \mathbb{N}$ and let $x_{-n},\ldots ,x_m$ verify $x_{-n}=x$ and $x_m=y$. By definition of the Lax-Oleinik semi-group, we have 
$$\T{m} u(y)\leqslant u(x_0)+\sum_{i=0}^{m-1}c(x_i,x_{i+1}),$$
and similarly,
$$\TT{n} u(x)\geqslant u(x_0)-\sum_{i=-n}^{-1}c(x_i,x_{i+1}).$$
Putting these two inequalities together, we find that 
$$\T{m}u(y)-\TT{n}u(x)\leqslant \sum_{i=-n}^{m-1}c(x_i,x_{i+1}).$$
Since the chain between $x$ and $y$ was taken arbitrarily, we obtain 
$$\T{m}u(y)-\TT{n}u(x)\leqslant c_{n+m}(x,y).$$
If $n'>n$, since $u\< c+\alpha[0]$ we have that 
$$\TT{n}u-n\alpha[0]\geqslant \TT{n'}u-n'\alpha[0].$$
Therefore the following hold
$$
\T{m}u(y)-\TT{n}u(x)\leqslant\T{m}u(y)-\TT{n'}u(x).$$
Therefore,
$$
\T{m}u(y)-\TT{n}u(x)+(m+n)\alpha[0]\leqslant c_{n'+m}(x,y)+(m+n')\alpha[0].
$$
Finally, letting $n'$ go to infinity and taking the liminf, we obtain
\begin{eqnarray*}
\T{m}u(y)-\TT{n}u(x)+(m+n)\alpha[0]&\leqslant&\liminf_{n'\to +\infty}c_{n'+m}(x,y)+(m+n')\alpha[0]\\
 &\leqslant&h(x,y).
\end{eqnarray*}
\end{proof}

 An easy consequence of the previous theorem is that whenever the
 function $h$ is finite, then if $u$ is a critically dominated
 function, the sequences $\T{n}u+n\alpha [0]$ and $\TT{n}u-n\alpha
 [0]$ are both simply bounded since they are respectively non decreasing and non increasing and therefore converge to respectively $u_-$ and $u_+$. Moreover, by equi-continuity (\ref{HH}), the convergences are uniform on compact
 subsets. Therefore, by continuity of the semi-groups for the compact open topology (see \ref{HH}), $u_-$ is
 a negative weak KAM solution and $u_+$ is a positive weak KAM
 solution.
Let us state a well known and useful lemma (cf. \cite{Con}):
\begin{lm}
 Let $(u_{\alpha})_{\alpha \in A}$ be a family of critically dominated functions. Let 
$$u=\inf_{\alpha \in A}u_{\alpha},$$
 this function is either identically $-\infty$ either it is finite everywhere.
 Moreover if $u$ is finite, the following relation holds:
$$\ttt \inf_{\alpha \in A}u_\alpha =\inf_{\alpha \in A}\ttt u_\alpha.$$
If furthermore  the $u_\alpha$ are weak KAM solutions and if the function $u$ is not identically $-\infty$ then it is a weak KAM solution.
\end{lm}
\begin{proof}
 The fact that $u$ is either identically $-\infty$ or everywhere
 finite comes from the fact that the domination hypothesis is stable by taking an infimum, therefore,
$$\forall (x,y)\in X,u(y)\leqslant u(x)+c(x,y)+\alpha[0].$$ 
Assume now that $u$
 is finite.
The following holds:
\begin{eqnarray*}
\ttt u(x) &=&\inf_{y\in X}u(y)+c(y,x)\\
&=&\inf_{y\in X}\inf_{\alpha\in A}u_\alpha(y)+c(y,x)\\
&=&\inf_{\alpha\in A}\inf_{y\in X}u_\alpha(y)+c(y,x)\\
&=&\inf_{\alpha\in A}\ttt u_\alpha(x).
\end{eqnarray*}
If moreover the $u_\alpha$ are weak KAM solutions, the following holds:
\begin{eqnarray*}
 \ttt u(x)+\alpha[0]&=&\inf_{\alpha\in A}\inf_{y\in X}u_\alpha(y)+c(y,x)+\alpha[0]\\
&=&\inf_{\alpha\in A}\ttt u_\alpha(x)+\alpha[0]\\
&=&\inf_{\alpha\in A}u_\alpha(x)=u(x).
\end{eqnarray*}

\end{proof}
As a consequence, still in the case when $h$ is finite, we have the
following theorem which first part was already established.
\begin{Th}
Assume $h$ is finite. Let $u\<c+\alpha [0]$ be a dominated function, then the sequences
 $\T{n}u+n\alpha [0]$ and $\TT{n}u-n\alpha [0]$ converge respectively
 to $u_-$ and $u_+$, a negative weak KAM solution and a positive weak
 KAM solution. Moreover, the functions $u_+$ and $u_-$ verify the
 following properties:
$$u_-=\inf_{w_-\geqslant u}w_-$$ 
where the infimum is taken over negative weak KAM solutions.
$$u_+=\sup_{w_+\leqslant u}w_+$$
 where the supremum is taken over positive weak KAM solutions.
\end{Th}
\begin{proof}
 Let us consider the function $u'$ defined by
$$u'=\inf_{w_-\geqslant u}w_-.$$ 
First notice that the set $\{w_-,w_-\geqslant u\}$ such that $w_-$ is a weak KAM solution is not empty because $u_-$ belongs to it.
The previous lemma shows that $u'$ is
 a negative weak KAM solution. Moreover, we have the following inequality: 
 $$\T{n}u+n\alpha
 [0]\leqslant \T{n}u'+n\alpha [0]=u'.$$
  Since the sequence
 $\T{n}u+n\alpha [0]$ converges to the weak KAM solution $u_-$ which is smaller or equal to $u'$, we have in fact $u_-=u'$.
The proof for the time positive case is the same.
\end{proof}

We now give a representation formula for the
function $h$:
\begin{Th}\label{Peierl}
 The Peierls barrier satisfies
$$\forall x,y\in X, h(x,y)=\sup_{\substack{u\<c+\alpha [0]\\n,m\in
     \mathbb{N}}}\T{n}u(y)-\TT{m}u(x)+(n+m)\alpha[0].$$
\end{Th}
\begin{proof}
 One inequality has been proved in \ref{inegh}, therefore, we only have to find a dominated function which realizes the case of equality. We have already seen (\ref{vanish}) that
\begin{equation}\label{phi1}
\TTT \varphi_{1,x}(x)-\alpha[0]= 0.
\end{equation}
Now using the fact that the sequence of functions 
$$\T{n}\varphi_{1,x}+n\alpha[0]=\varphi_{n+1,x}$$
 converge to $h_x$ we obtain that 
\begin{equation}\label{phi2}
\lim_{n\to +\infty}\T{n}\varphi_{1,x}-\TTT \varphi_{1,x}(x)+(n+1)\alpha[0]= h(x,y).
\end{equation}
This ends the proof.
\end{proof}
\begin{co}\label{0}
 For all positive integer $m$ we have that 
 $$\TT{m} \varphi_{1,x}(x)-m\alpha[0]= 0.$$
For all integer $m$ we have $\TT{m} \varphi_{x}(x)-m\alpha[0]= 0$.
Moreover, the following hold
$$\lim_{n\to +\infty}\T{n}\varphi_{1,x}(y)-\TTT \varphi_{1,x}(x)+(n+1)\alpha[0] = h(x,y),$$
$$\lim_{n\to +\infty}\T{n}\varphi_{x}(y)-\varphi_{x}(x)+n\alpha[0] = h(x,y).$$
\end{co}
\begin{proof}
 Using \ref{phi1}, and the fact that $\varphi_{1,x}$ is a critical sub-solution, we get the following generalization of \ref{phi2}:
$$\forall m\in\mathbb{N}^*,\lim_{n\to +\infty}\T{n}\varphi_{1,x}(y)-\TT{m} \varphi_{1,x}(x)+(m+n+1)\alpha[0]\geqslant h(x,y).$$
Once again, this inequality is in fact an equality (by \ref{inegh}).
Now using again the fact that the sequence of functions 
$$\T{n}\varphi_{1,x}+n\alpha[0]=\varphi_{n+1,x}$$ converge to $h_x$ we obtain that $\TT{m} \varphi_{1,x}(x)-m\alpha[0]= 0$.

To prove the second point, notice that by \ref{egalitetotale} and $\varphi_{1,x}\geqslant \varphi_x$ we get that for all $m>0$ and $n\in\mathbb{N}$,
$$
\T{n}\varphi_{x}(y)-\TT{m} \varphi_{x}(x)+\alpha[0] \geqslant \varphi_{n-1,x}(y)-\TT{m} \varphi_{1,x}(x).
$$
Therefore we have
$$\lim_{n\to +\infty}\T{n}\varphi_{x}(y)-\TT{m} \varphi_{x}(x)+(m+n+1)\alpha[0] \geqslant h(x,y).$$
By \ref{inegh}, these inequalities are in fact equalities which implies that  for all integer $m$ we have $\TT{m} \varphi_{x}(x)-m\alpha[0]= 0$.
\end{proof}
We are now able to give the proof of \ref{A}:
\begin{proof}[Proof of \ref{A}]
  We know that if $u$ is a critically dominated function and $(x_n)_{n\in \mathbb{Z}}$ is a calibrated sequence for $u$, then for all $n\in \mathbb{N}$, we have (\ref{IIII}) 
$$\T{n}u(x_0)+n\alpha [0]=\TT{n}u(x_0)-n\alpha [0]=u(x_0).$$
Therefore if $h$ is identically $+\infty$, then there are no calibrated bi-infinite chains for the critically dominated function $\varphi_{1,x}$ where $x$ is any point of $X$ (the sequence $\T{n}\varphi_{1,x}(x_0)+n\alpha [0]$ goes to $+\infty$ and therefore may not be constant) which proves that in this case, $\AA=\varnothing$ and at the same time that  $\A=\varnothing$.

When $h$ is finite, by \ref{Peierl} and \ref{IIII}, $h(x,x)=0$ if and only if for any critically dominated function $u$, the sequences 
$$\T{n}u(x)+n\alpha[0]\; \textrm{and}\; \TT{m}u(x)-m\alpha[0]$$
 are constantly equal to $u(x)$. Assume now that $u$ is the function given by \ref{I}. Applying, \ref{IIII} we obtain that $x\in \A_u=\A$.
\end{proof}
Let us now point out a phenomenon that is of some resemblance
with  paired weak KAM solutions in the compact case
(\cite{Fa}).
\begin{pr}
 Assume that $h$ is finite. Let $u$ be a critically dominated
 function. Let $u_-$ be the limit of the sequence of functions
 $\T{n} u+n\alpha [0]$ (it is a negative weak KAM solution). Let
 $u_{-+}$ be the limit of the sequence of functions $\TT{n} u_-
 -n\alpha [0]$ which is a positive weak KAM solution. Then, again let
 $u_{-+-}$ be the limit of the $\T{n}u_{-+} +n\alpha [0]$ and $u_{-+-+}$ be the limit
 of the $\TT{n} u_{-+-} -n\alpha [0]$.\\ Then,
 $u_{-+}=u_{-+-+}$.
\end{pr}
\begin{proof}
 We have seen that
\begin{eqnarray*}
u_-&=&\inf_{w_-\geqslant u}w_-\\
u_{-+}&=&\sup_{w_+\leqslant u_-}w_+\\
u_{-+-}&=&\inf_{w_-\geqslant u_{-+}}w_-\\
u_{-+-+}&=&\sup_{w_+\leqslant u_{-+-}}w_+,
\end{eqnarray*}
where $w_-$ and $w_+$ denote each time respectively negative and
positive weak KAM solutions.  Obviously, since $u_{-+}\leqslant u_{-+-}$,
 by the above formula $u_{-+}\leqslant u_{-+-+}$. We also have
$u_-\geqslant u_{-+-}$. Therefore, by monotonicity of the Lax-Oleinik
semi-group we obtain $u_{-+}\geqslant u_{-+-+}$ which gives the desired equality.
\end{proof}
\begin{rem}\rm
In other words, the operation which sends a sub-solution $u$ to the weak KAM solution $u_{-+}$ is idempotent. This is comparable to the result we obtained in \ref{LO}.\\
 The assumption that the Peierls barrier is finite is rather
 strong in the non compact case. To ensure that the sequence $\T{n} u+n\alpha [0]$
 (resp. $\TT{n} u-n\alpha [0]$) converges, it is enough to suppose
 that there is a negative (resp. positive) weak KAM solution that is
 greater (resp. smaller) than $u$.
\end{rem}
We conclude by showing that the function $\varphi$ may help solving the question of the finiteness of the Peierls barrier $h$.
\begin{pr}
 The following statements are equivalent:
\begin{enumerate}
 \item the Peierls barrier is finite,
\item there is an $(x,y)\in X^2$ such that the sequence $\T{n}
  \varphi_x(y)+n\alpha[0]$ is bounded,
\item there is an $x\in X$ such that the sequence $\T{n}
  \varphi_x+n\alpha[0]$ is point-wise bounded,
\item for every $x\in X$, the sequence $\T{n} \varphi_x+n\alpha[0]$ is
  point-wise bounded,
\item for all $u$ critically dominated, the sequences $(\T{n}u+n\alpha[0])_{n\in \mathbb{N}}$ and $(\TT{n}u-n\alpha[0])_{n\in \mathbb{N}}$ converge uniformly on compact sets to respectively a negative weak KAM solution and a positive weak KAM solution.
\end{enumerate}
\end{pr}
\begin{proof}
It suffices to notice that by (\ref{egalitetotale}) we have $\varphi_{1,x}=\ttt \varphi_x +\alpha[0]$ for all $x\in X$. Hence applying (\ref{semi}) we obtain
 $$\T{n} \varphi_x+n\alpha[0]=\varphi_{n-1,x}.$$ 
 Therefore, this sequence of functions converges uniformly on all compact sets to $h_x$ which is either everywhere finite or everywhere $+\infty$.
The last point is a direct consequence of \ref{inegh}.
 \end{proof}

\appendix

\section{Appendix: Existence of weak KAM solutions}
What comes in the following section is mostly adapted from
\cite{FaMa}. Let us consider a metric space $X$ such that its closed
balls are compact and, which verifies the following: 
\begin{df}\label{scale}\rm
Given constants $K\in \R{}$, $B\geqslant 1$ we will say the metric space $X$ is a $B$-length space at scale $K$ if for every
$(x,y)\in X^2$, there exist $\left(x=x_0,\ldots ,x_n=y\right)\in X^{n+1}$ such that for all $i\leqslant n-1$, $\d (x_i,x_{i+1})\leqslant K$ and,
$\sum_{0\leqslant i\leqslant n-1} \d(x_i,x_{i+1})\leqslant B\d (x,y)$ where $\d$
denotes the distance function.
\end{df}

We start with a simple but fundamental lemma:

\begin{lm}\label{number}
 If $X$ is a  $B$-length space at scale $K$ then for every $(x,y)\in X^2$, there exist $\left(x=x_0,\ldots ,x_n=y\right)\in X^{n+1}$ such that for all $i\leqslant n-1$, $\d (x_i,x_{i+1})\leqslant K$ and,
$\sum_{0\leqslant i\leqslant n-1} \d(x_i,x_{i+1})\leqslant B\d (x,y)$ and 
$$n\leqslant \frac{2B\d(x,y)}{K}+1.$$
\end{lm}
\begin{proof}
 Let us take a chain $(x=x_0,\ldots ,x_n=y)$ verifying the hypothesis of \ref{scale} and such that $n$ is minimal. Necessarily, 
$$\forall i\leqslant n-2, \d(x_i,x_{i+1})+\d(x_{i+1},x_{i+2})\geqslant K,$$
for otherwise, the same sequence without $x_{i+1}$ would itself verify the hypothesis of \ref{scale}.

Therefore, if we call $m=\lfloor n/2\rfloor$ then $n\leqslant 2m+1$ and $mK\leqslant B\d(x,y)$.
\end{proof}

\begin{ex}\rm
 \begin{enumerate}
  \item A metric compact space $C$ is a 1-length space at scale $\diam (C)$,
\item a length space is a 1-length space at scale $K$ for every $K>0$,
\item a graph endowed with its graph metric is a 1-length space at scale 1,
\item a grid $G_{\e}=\e\mathbb{Z}^n\subset \mathbb{R}^n$ endowed with the metric induced by the inclusion in $\R{n}$ is a $\sqrt{n}$-length space at scale $\e$,
\item if a metric space, $(X,d)$, whose closed balls are compact is a $B$-length space at scale $K$ for every $K>0$ then it is Lipschitz equivalent to a length space,
\item the set $\mathcal{P}$ of prime numbers endowed with the distance $\d(p,p')=|p-p'|$ is not a length space at any scale.
 \end{enumerate}
\end{ex}
\begin{proof}
 Items 1, 2,3,4 and 6 are clear. The proof of 5 uses standard ideas in topology and in the study of length spaces (see for example \cite{gro}, Theorem 1.8). Let $(x,y)\in X^2$ be two distinct points. We want to construct a continuous curve from $x$ to $y$ which metric length is less than $B\d(x,y)$. Applying that $X$ is a $B$-length space at scale $1/n$ we find for any $n\geqslant 1$ a sequence of points
 $\left(x=x_0^n,\ldots ,x_{N_n}^n=y\right)\in X^{N_n+1}$ such that for all $i\leqslant N_n-1$ we have $\d (x_i^n,x_{i+1}^n)\leqslant 1/n$ and,
\begin{equation}\label{bor}
\sum_{0\leqslant i\leqslant N_n-1} \d(x_i^n,x_{i+1}^n)\leqslant B\d (x,y).
\end{equation}
Moreover, it is clear that the sequence $N_n$ goes to $+\infty$ and by \ref{number},  we can assume that for $n$ large enough, the following holds: 
\begin{equation}\label{bor1}
\forall n\in \mathbb{N}^*, N_n\leqslant  2nB\d(x,y)+1\leqslant 3nB\d(x,y).
\end{equation}
Clearly, we also have:
\begin{equation}\label{borne}
\forall n\in \mathbb{N}^*,\forall i\leqslant N_n, \d(x,x_i^n)\leqslant B\d(x,y).
\end{equation}

 We define for any integer $n$ and $i\leqslant N_n$, $f_n(i/N_n)=x_i^n$.
 For any integer $n$ large enough and any $i,j\leqslant N_n$, the following holds :
 \begin{equation}\label{lip}
 \d(f_n(i/N_n),f_n(j/N_n))\leqslant \frac{|j-i|}{n}\leqslant 3B\d(x,y)\frac{|j-i|}{N_n}. 
 \end{equation}
  Let $(q_k),k\in \mathbb{N}$ be a dense sequence in $[0,1]$. For any $k\in \mathbb{N}$ let us choose a sequence $(a_n^k=i_n^k/N_n),n\in \mathbb{N}$ which converges to $q_k$, where $i_n^k$ is always smaller than $N_n$. Up to doing a diagonal extraction, using \ref{borne}, we can assume that all the sequences $(f_n(a_n^k),n\in \mathbb{N})$ converge to an element $x_k$ of $X$. Let us define 
  $$\forall k\in \mathbb{N}, f(q_k)=x_k.$$
  By \ref{lip}, we have for $n$ large enough, 
  $$\d(f_n(a_n^k),f_n(a_n^{k'}))\leqslant 3B\d(x,y)|a_n^k-a_n^{k'}|,$$
  therefore, letting $n$ go to $+\infty$ we obtain
  $$\forall (k,k')\in \mathbb{N}^2, \d(f(q_k),f(q_{k'}))\leqslant 3B\d(x,y)|q_k-q_{k'}|.$$
  Since $(q_k)_{k\in\mathbb{N}}$ is dense in $[0,1]$, $X$ is complete and by the previous inequalities $f$ is uniformly continuous (it is in fact Lipschitz), we can extend it to a continuous function, that we will still call $f$, from $[0,1]$ to $X$. Finally, by \ref{bor}, the length of $f$ is smaller than $B\d(x,y)$.

Let us now denote $\d_l$ the distance on $X$ induced by its metric length structure. More precisely, if $x,y$ are two points, $\d_l(x,y)$ is nothing but the infimum of the length of a path joining $x$ to $y$ over all such paths (see \cite{gro} (p. 2 and 3) for a more precise definition). By the above construction, the space $(X,\d_l)$ is a length space and the application identity from $(X,\d_l)$ to $(X,\d)$ is $B$-Lipschitz. Moreover, by definition of $\d_l$, it's inverse from $(X,\d)$ to $(X,\d_l)$ is 1-Lipschitz.
\end{proof}
 A complete, connected
Riemannian manifold is a 1-length space at scale $K$ for all $K>0$ so this property will clearly
hold. Assume from now on that $X$ is a $B$-length space at scale $K$ for some $(B,K)$.

  Let
$c: X\times X \rightarrow \R{}$ be a continuous function which
verifies the conditions of uniform super-linearity (\ref{unif}) and  uniform boundedness (\ref{unifb}) stated in the introduction.
We recall that a function $u:X\rightarrow \R{}$ is an
$\alpha$-sub-solution or that it is dominated by $c+\alpha$ (in short
$u\<c+\alpha$) if for every $(x,y)\in X^2$ we have $u(x)-u(y)\leqslant
c(y,x)+\alpha$ (see \ref{sub} in the introduction). We will denote by $\HH(\alpha)$ the set of such
functions.
\\
 Finally, let us state the definitions of the well known Lax-Oleinik
semi-groups:
\\
 for a function $u:X\rightarrow \overline{\R{}}$ we
define the function
$$\ttt u:X\rightarrow \overline{\R{}}\ \ \mathrm{
by}\ \ 
\ttt u(x)=\inf_{y\in X}\left\{ u(y)+c(y,x)\right\},$$
$$\TTT u:X\rightarrow \overline{\R{}}\ \ \mathrm{
by}\ \ 
\TTT u(x)=\sup_{y\in X}\left\{ u(y)-c(x,y)\right\}.$$
The following lemma is not difficult to check.
\begin{lm}\label{monotonicity}
 If $k\in \R{}$ and $u:X\rightarrow \overline{\R{}}$ then $\ttt
 (u+k)=k+\ttt u$ that is, the Lax-Oleinik semi-group commutes with the addition of constants.  Moreover, if $v:X\rightarrow \overline{\R{}}$ is
 another function such that $u\leqslant v$ then $\ttt u\leqslant \ttt
 v$, in other words, the semi-group is monotonous.
\end{lm}
\begin{df}
 Let $(k,b)\in \R{2}$, we will say that $f:X\rightarrow \R{}$ is $(k,b)$-Lipschitz in the large or $f\in \mathrm{Lip}_{(k,b)}(X,\R{})$  if
$$\forall (x,y)\in X^2, |f(x)-f(y)|\leqslant k\d(x,y)+b.$$
\end{df}
\begin{ex}
Bounded functions are Lipschitz in the large.

 Uniformly continuous functions on a length space are Lipschitz in the large.
\end{ex}
Although functions Lipschitz in the large are not necessarily continuous, obviously they satisfy the following lemma:
\begin{lm}\label{bounded}
 A function Lipschitz in the large is bounded on any ball of finite radius.
\end{lm}
These functions give a nice setting to apply the Lax-Oleinik semi-groups as shown in the following propositions:
\begin{pr}
The following properties hold:
 \begin{enumerate}
  \item If $k\in \R{}$ and $u:X\rightarrow \R{}$ then $u\in \HH(\alpha)$ if
    and only if $u+k\in \HH(\alpha)$.
 \item If $u:X\rightarrow \R{}$ is $(k,b)$-Lipschitz in the large then $u\in \HH(C(k)+b)$.
\item The subset $\HH(\alpha)$ is convex and closed in the space $\mathcal{F}(X,\R{})$ of finite real valued functions on $X$ endowed with the topology of point-wise convergence.
  
\item If $\alpha\leqslant \alpha'$ then $\HH(\alpha) \subset \HH(\alpha')$.
\item If $\HH(\alpha)\neq \varnothing$ then $\alpha\geqslant \sup \{-c(x,x),x\in
  X\}\geqslant -A(0)$.
 \end{enumerate}
\end{pr}
\begin{proof}
 Statements (1) and (4) are direct consequences of the definitions.
If $u\in \mathrm{Lip}_{(k,b)}(X,\R{})$ then 
$$\forall (x,y)\in X^2,u(x)-u(y)\leqslant k\d(x,y)+b\leqslant
c(y,x)+C(k)+b,$$
 which proves statement (2).  

To prove statement (3),
just notice that $\HH(\alpha)$ is an intersection of closed half spaces
for the given topology, one for each couple of points of
$X$.

 As for statement (5), observe that if $u\in \HH(\alpha)$ and $x\in
X$ then
$$0=u(x)-u(x)\leqslant c(x,x)+\alpha,$$
 which implies (5).
\end{proof}
In the following, we will need this lemma:
\begin{lm}\label{lip}
 Let $\alpha\in \R{}$, then there exists constants $k(\alpha)$ and $b(\alpha)$ such that for any  $u$ which is $\alpha$-dominated,  then $u$ is Lipschitz in the large with constants $k(\alpha)$ and $b(\alpha)$.
\end{lm}
\begin{proof}
 Let us consider $u\in \HH (\alpha)$ and $x_0\in X$. Then one has 
$$\forall y\in X, u(x_0)-u(y)\leqslant c(y,x_0)+\alpha\leqslant
A(\d(y,x_0))+\alpha$$ where we have used first the domination of $u$ and
then the uniform boundedness of $c$. Moreover, using the assumption we
made on the metric $\d$ and \ref{number}, the constants $K$, $B$  satisfy that for any $y\in X$,
$$ u(x_0)-u(y)\leqslant (A(K)+\alpha)\left(
\frac{2B\d(x_0,y)}{K}+1\right) $$ 
which proves that $u\in \mathrm{Lip}_{2(A(K)+\alpha)B/K,A(K)+\alpha}(X,\R{})$.
\end{proof}

\begin{pr}\label{HH} The following properties are verified:
 \begin{enumerate}
  \item[(i)] Let $u:X\rightarrow \R{}$ be a function. We have $u\<c+\alpha$ if
    and only if $u\leqslant \ttt u+\alpha$.
\item[(ii)] The following holds:
$$\ttt \left(\mathrm{Lip}_{(k,b)}(X,\R{})\right) \subset \hh{C(k)+b}.$$
Moreover,  the set of functions $\ttt \left(\mathrm{Lip}_{(k,b)}(X,\R{})\right)$ is locally equi-continuous.
Finally the mapping $\ttt$ restricted to $\mathrm{Lip}_{(k,b)}(X,\R{})$ is continuous for the topology of uniform convergence on compact subsets.
\item[(iii)] The map $\ttt$ sends $\HH(\alpha)$ into $\hh{\alpha}$ and is continuous for the topology of uniform convergence on compact subsets. Moreover,  the set of functions $\ttt \left(\HH{(\alpha)}\right)$ is locally equi-continuous.
\item[(iv)] \label{atteint} If $u\in \mathrm{Lip}_{(k,b)}(X,\R{})$ is lower semi-continuous, then for every $x\in
  X$, there is a $y\in X$ such that $\ttt u(x)=u(y)+c(y,x)$.
 \end{enumerate}
\end{pr}
\begin{proof}
 To prove (i), remark that domination of $u$ by $c+\alpha$ is equivalent to
$$\forall (x,y)\in X^2 ,u(x)\leqslant u(y)+c(y,x)+\alpha,$$
which is equivalent to 
$$\forall x\in X ,u(x)\leqslant \inf_{y\in X} u(y)+c(y,x)+\alpha,$$
 but the right hand side is precisely $\ttt u(x)+\alpha$.
\\Let us prove (ii). Let $u\in \mathrm{Lip}_{(k,b)}(X,\R{})$ and let $x_0\in X$ and $r>0$. We know that
$$\forall y\in X,\forall x\in B(x_0,r),u(y)+c(y,x)\geqslant c(y,x)+u(x_0)-k\d(x_0,y)-b,$$
therefore, using the super-linearity of $c$ we get that
\begin{eqnarray}
u(y)+c(y,x)&\geqslant& 2k\d(x,y)+C(2k)+u(x_0)-k\d(x_0,y)-b\nonumber \\
&\geqslant& k\d(x_0,y)-2kr+C(2k)+u(x_0)-b. \label{1}
\end{eqnarray}
Now, by definition of the Lax-Oleinik semi-group,
$$\ttt u(x)=\inf_{y\in X} u(y)+c(y,x)\leqslant
u(x_0)+c(x_0,x)\leqslant u(x_0)+A(r),$$
 so by condition (\ref{1}) it is
not restrictive to take the infimum on points at a distance less than
$D(r,k,b)=(A(r)+2kr-C(2k)+b)/k$ from $x_0$. Using that $u$ (by lemma \ref{bounded}) and $c$ (by continuity) are bounded below on balls of finite radius (which are compact), the infimum in the Lax-Oleinik semi-group is
finite and if reached, can only be reached in $\overline{B}(x_0,D(r,k,b))$. Note that this already
proves (iv) because a lower semi-continuous function achieves its
minimum on a compact set. The constant $D(r,k,b)$ is
independent of $x\in B(x_0,r)$ and $u\in \mathrm{Lip}_{(k,b)}(X,\R{})$. Therefore if $x_1\in \overline{B}(x_0, r)$
then in the definition of $\ttt u(x_1)$ the infimum may be taken on
points which lie in $\overline{B}(x_0,D(r,k,b))$. Since $\overline{B}(x_0,D(r,k,b))\times
\overline{B}(x_0,r)$ is compact, the restriction of $c$ to this
domain is uniformly continuous, let $\om$ be a modulus of
continuity of $c$ on that domain. One has that the restriction of $\ttt u$ to
$\overline{B}(x_0,r)$ is a finite infimum of equi-continuous
functions and is therefore itself continuous with same modulus of
continuity which only depends on $c$, so the family $\ttt (\mathrm{Lip}_{(k,b)}(X,\R{}))$ is
in fact locally equi-continuous. 
\\Now that we know it is finite, let us check that $\ttt u$ is
$(C(k),+b)$-dominated. This is in fact a direct consequence of the monotonicity of the
Lax-Oleinik semi-group (\ref{monotonicity}). In fact, by (i), since
$u\<c+C(k)+b$ it follows that $u\leqslant \ttt u+C(k)+b$. We therefore have that $\ttt
u\leqslant \ttt (\ttt u)+C(k)+b$ which proves that $\ttt u\<c+C(k)+b$. 
\\It remains to prove that the restriction of this mapping to $\mathrm{Lip}_{(k,b)}(X,\R{})$ is continuous for the topology of uniform convergence on compact subsets. Let $v\in \mathrm{Lip}_{(k,b)}(X,\R{})$ be another dominated
function, $x\in X$. Let $\e>0$ and $x_1\in X$ be such that 
$$|\ttt u(x)-u(x_1)-c(x_1,x)|<\e,$$
 and similarly, chose $x_2$ such that 
$$|\ttt v(x)-v(x_2)-c(x_2,x)|<\e.$$
 Note that both $x_1$ and $x_2$ are necessarily
in $\overline{B}(x,D(0,k,b))$. The following inequality holds:
\begin{eqnarray*}
\ttt v(x)-\ttt u(x)&\leqslant&
v(x_1)+c(x_1,x)-u(x_1)-c(x_1,x)+\e\\
&\leqslant&
\sup_{\overline{B}(x,D(0,k,b))}|u-v|+\e.
\end{eqnarray*} 
By a symmetrical argument, we also
have
$$\ttt u(x)-\ttt v(x)\leqslant u(x_2)+c(x_2,x)-v(x_2)-c(x_2,x)\leqslant \sup_{\overline{B}(x,D(0,k,b))}|u-v|+\e.$$
This being true for all $\e>0$, we have just proved that if $A\subset X$ is compact, then 
 $$\sup_{A}|\ttt u-\ttt v|\leqslant \sup_{A_{D(0,k,b)}}|u-v|,$$
  where $A_{D(0,k,b)}=\{ x\in X,\d(A,x)\leqslant D(0,k,b)\}$ is still compact because it is contained in a ball of finite large radius. This achieves the proof of (ii).\\
To prove (iii), note that by lemma \ref{lip}, dominated functions are equi-Lipschitz in the large. Therefore the family of functions in $\ttt (\HH(\alpha))$ is locally equi-continuous.
\end{proof}
As an immediate consequence of the previous proof we deduce the following result:
\begin{lm}[a priori compactness]\label{apriori}
 Given constants $k$, $b$, $\e>0$ and a compact set $A\subset X$ there is a compact set $A'\subset X$ such that if $v\in \mathrm{Lip}_{(k,b)}(X,\R{})$, $x\in A$ then
$$|u(y)+c(y,x)-\ttt u(x)|\leqslant \e \Longrightarrow y\in A'.$$
\end{lm}
We now can prove the weak KAM theorem:
\begin{proof}[Proof of theorem \ref{kam}]\label{proof kam}
First, notice that saying that $\HH(\alpha)$ is empty is equivalent to
saying that $\hh{\alpha}$ is empty, because of part (iii) of the previous
proposition (\ref{HH}). Let $\mathbbm{1}$ be the constant function equal to $1$
on $X$ and let $\widehat{C^0(X,\R{})}$ be the quotient of
$C^0(X,\R{})$ by the subspace of constant functions $\R{}\mathbbm{1}$
and let $q$ be the projection operator. Since the semi-group $\ttt$
commutes with the addition of constants, it induces a semi group on
$\widehat{C^0(X,\R{})}$ that we will denote $\tttt$. The topology on
$\widehat{C^0(X,\R{})}$ is the quotient of the compact open topology
on $C^0(X,\R{})$, which makes it a locally convex vector space.\\
 We will call $\HHH (\alpha)$ the image $q(\hh{\alpha})$. It is convex because so is
$\hh{\alpha}$ . Let us introduce the subset $C^0_{x_0}$ of $C^0(X,\R{})$
consisting of the functions which vanish at $x_0$, where $x_0$ is any
point of $X$. Then, $q$ induces a homomorphism of $C^0_{x_0}$ onto
$\widehat{C^0(X,\R{})}$. Since $\hh{\alpha}$ is stable by addition of
constants, $\HHH (\alpha)$ is also the image by $q$ of the set $\HH(\alpha)\cap
C^0_{x_0}=\HH_{x_0}(\alpha)$. Now, $\HH_{x_0}(\alpha)$ is closed for the compact open
topology, it consists of functions which all vanish at $x_0$. We have
seen in the proof of \ref{HH} that $\ttt (\HH(\alpha))$ is a family of
locally equi-continuous and equi-Lipschitz in the large, therefore locally equi-bounded functions. By the Ascoli theorem, we deduce that
$\ttt(\HH_{x_0}(\alpha))$ is relatively compact. Furthermore, since
$$\tttt(q(u))=q(\ttt u-\ttt u(x_0)\mathbbm{1}),$$
we obtain that
$$\tttt(\HHH(\alpha))=q(\ttt(\HH_{x_0}(\alpha)))$$ 
is also relatively compact and the
closed convex envelope of $\tttt(\HHH (\alpha))$ that we will denote $H(\alpha)$
is compact. Note also that  $H(\alpha)\subset \HHH (\alpha) $, since $\HHH (\alpha)$ is convex, closed for the compact open topology and it contains $\ttt(\HHH (\alpha))$.
\\ As a first consequence, if 
$$\alpha [0]=\inf\{\alpha\in
\R{},\HH(\alpha)\neq \varnothing\},$$
 then $\bigcap_{\alpha>\alpha
  [0]}H(\alpha)\neq\varnothing$ as the intersection of a decreasing family
of compact nonempty sets. It follows that $\HH(\alpha [0])$ is non
empty for it contains $q^{-1}\left(\bigcap_{\alpha>\alpha
  [0]}H(\alpha)\right)$.\\ Finally, it is obvious that $\tttt$ carries
$H(\alpha)$ into itself. Since this last subset is a nonempty convex
compact subset of a locally convex topological vector space, we can
apply the Schauder-Tykhonov theorem  (\cite{dug} p.414, Theorem 2.2). This gives that $\tttt$ has a
fixed point in $H(\alpha)$ as soon as $\HH(\alpha)\neq \varnothing$ that is for
all values $\alpha\geqslant \alpha [0]$.\\ If we call $q(u)$ such a fixed
point, with $u\in \HH(\alpha [0])$, we see there is a constant $\alpha'$
such that $\ttt u=u+\alpha'$. Obviously, $u\<c-\alpha'$ so $-\alpha'\geqslant \alpha
[0]$. Moreover since $u\in \HH(\alpha [0])$ we must have $u\leqslant
\ttt u+\alpha [0]$ which gives $u=\ttt u-\alpha'\leqslant \ttt u+\alpha [0]$
and $-\alpha'\leqslant \alpha [0]$. We therefore conclude that $-\alpha'=\alpha
[0]$.
\end{proof}

\bibliography{C1-04-08}
\bibliographystyle{alpha}
\end{document}